\documentclass[12pt,twoside,reqno]{amsart}
\usepackage[colorlinks=true,citecolor=blue]{hyperref}
\usepackage{mathptmx, amsmath, amssymb, amsfonts, amsthm, mathptmx, enumerate, color,mathrsfs}
\usepackage{graphicx}
\usepackage{multirow}
\usepackage{epstopdf}
\usepackage{multicol}
\usepackage{algorithm}
\usepackage{algorithmicx}  
\usepackage{epstopdf}
\usepackage{extarrows}
\usepackage{verbatim}
\usepackage{booktabs}
\usepackage{algpseudocode}  
\usepackage{listings}
\usepackage{diagbox}
\usepackage{xcolor}
\usepackage{indentfirst}
\usepackage{subfigure} 

\newtheorem{theorem}{Theorem}[section]
\newtheorem{lemma}[theorem]{Lemma}

\theoremstyle{definition}
\newtheorem{definition}[theorem]{Definition}

\newtheorem{remark}[theorem]{Remark}
\numberwithin{equation}{section}

\newcommand{\Div}{\mbox{{\rm div}\,}}
\newcommand{\bu}{\mathbf{u}}

\newcommand{\bv}{\mathbf{v}}

\newcommand{\bF}{\mathbf{F}}
\newcommand{\bG}{\mathbf{G}}
\newcommand{\bW}{\mathbf{W}}
\newcommand{\bw}{\mathbf{w}}

\newcommand{\bL}{\mathbf{L}}

\newcommand{\bH}{\mathbf{H}}

\newcommand{\bUh}{\mathbb{U}_h}
\newcommand{\bVh}{\mathbb{V}_h}
\newcommand{\bphi}{\boldsymbol{\phi}}
\newcommand{\bpsi}{\boldsymbol{\psi}}
\newcommand{\bPhi}{\boldsymbol{\Phi}}

\newcommand{\td}{\text{\rm d}}
\newcommand{\tD}{\text{\rm D}}
\newcommand{\inttn}{\int_{t_n}^{t_{n+1}}}
\newcommand{\Rh}{\cal{R}_h}
\newcommand{\Ph}{\cal{P}_h}

\def\mbf#1{\mathbf{#1}}
\def\bb#1{\mathbb{#1}}
\def\cal#1{\mathcal{#1}}
\def\exc#1{\bb{E}\left[#1\right]}
\def\norm#1{\left\lVert{#1}\right\rVert}
\def\normbf#1{\left\lVert{#1}\right\rVert^{2}_{\mbf{L}^2}}
\def\normbfhalf#1{\left\lVert{#1}\right\rVert_{\mbf{L}^2}}
\def\normbfp#1{\left\lVert{#1}\right\rVert^{2p}_{\mbf{L}^2}}
\def\normbfh#1{\left\lVert{#1}\right\rVert^{2}_{\mbf{H}^1}}

\def\half{\frac{1}{2}}
\def\veps{\varepsilon}
\def\energy#1{J(\mbf{u}^{#1},\mbf{v}^{#1})}
\def\energyh#1{J(\mbf{u}_h^{#1},\mbf{v}_h^{#1})}
\def\normdel#1{\lambda\normbf{\Div(#1)}+\mu\normbf{\veps(#1)}}

\def\normdelbig#1{\lambda\normbf{\Div\big(#1\big)}+\mu\normbf{\veps\big(#1\big)}}

\def\normdelpbig#1{\lambda\normbfp{\Div\big(#1\big)}+\mu\normbfp{\veps\big(#1\big)}}

\def\innerprd#1#2{\lambda\big(\Div({#1}),\Div({#2})\big)+\mu\big(\veps({#1}),\veps({#2})\big)}
\def\innerprdminusk#1#2{-k\lambda\big(\Div({#1}),\Div({#2})\big)-k\mu\big(\veps({#1}),\veps({#2})\big)}
\def\innerlap#1#2{\big(\opL{#1}, {#2}\big)}

\def\bferr#1#2{\mbf{e}_\mbf{{#1}}^{#2}}

\def\engerr#1{J(\bferr{u}{#1},\bferr{v}{#1})}
\def\bfErr#1#2{\mbf{E}_\mbf{{#1}}^{#2}}

\def\Exp#1#2{#1\times 10^{#2}}
\newcommand{\opL}{\cal{L}}
\newcommand{\opLh}{\cal{L}_{h}}
\newtheorem{scheme}{Scheme}
 
\newcommand\commentone[1]{\textcolor{black}{#1}}

\begin{document}
\setcounter{page}{1}

\vspace*{2.0cm}
\title[Finite elements for nonlinear stochastic elastic wave equations]
{Fully discrete finite element methods for nonlinear stochastic elastic wave equations with multiplicative noise}
\author[Xiaobing Feng, Yukun Li, Yujian Lin]{Xiaobing Feng$^{1,*}$, Yukun Li$^2$, Yujian Lin$^3$}
\maketitle
\vspace*{-0.6cm}

\begin{center}
{\footnotesize

$^1$Department of Mathematics, The University of Tennessee,
Knoxville, TN 37996\\
$^2$Department of Mathematics, University of Central Florida, Orlando, FL 32816\\
$^3$Department of Mathematics, Northwestern Polytechnical University, Xian, Shaanxi, China, 710129\\
}\end{center}

\vskip 4mm {\footnotesize \noindent {\bf Abstract.}
This paper is concerned with fully discrete finite element methods for approximating variational solutions of nonlinear stochastic elastic wave equations with multiplicative noise. A detailed analysis of the properties of the weak solution is carried out and a fully discrete finite element method is proposed. Strong convergence in the energy norm with rate $\cal{O}(k+h^r)$ is proved, where $k$ and $h$ denote respectively the temporal and spatial mesh sizes, and $r(\geq 1)$ is the order of the finite element. Numerical experiments are provided to test the efficiency of proposed numerical methods and to validate the theoretical error estimate results.

}




\section{Introduction}\label{sec:introduction} 
This paper is concerned with numerical approximations of the following stochastic elastic wave equations with multiplicative noise of It\^o type:
\begin{alignat}{2}\label{eqn:elastic-wave-equation}
\bu_{tt}-\Div(\sigma(\bu))&=\bF[\bu]+\bG[\bu]\xi\quad &\text{in}&\  \cal{D}_{T}:=(0,T)\times\cal{D},\\
\bu(0,\cdot)&=\bu_{0},\quad\bu_{t}(0,\cdot)=\bv_{0}\quad &\text{in}&\ \cal{D},\label{20220526_1}\\
\bu(t,\cdot)&=0&\text{on}&\ \partial\cal{D}_{T}:=(0,T)\times\partial\cal{D},\label{20220526_2}
\end{alignat}
where 
$\bu_t=\frac{\td\bu}{\td t}$, $\xi = \dot{\bW} = \frac{\td \bW}{\td t}$ is the white noise, $\cal{D}\subset\bb{R}^{d}\ (d=2,3)$ is a bounded domain, $(\bu_{0},\bv_{0})$ is an $\bf{H}_{0}^{1}\times\bf{L}^{2}$-valued  random variable, and
\begin{align}
\sigma(\bu)&=(\lambda \Div\bu)\bf{I}+\mu\veps(\bu),\label{20220526_3}\\
\veps(\bu)&=\frac{1}{2}\bigl(\nabla\bu+(\nabla\bu)^{\bf{T}}\bigr),\label{20220526_4}\\
\bF[\bu]&=\bF(\bu,\nabla \bu),\label{20220526_5}\\
\bG[\bu]&=\bG(\bu,\nabla \bu).\label{20220526_6}
\end{align}
Here $\bf{I}$ denotes the unit matrix. $\bF[\bu]$ and $\bG[\bu]$ are two given nonlinear mappings satisfying some structure conditions. The multiplicative noise $\bG[\bu]\xi$ has the following three cases: 
\begin{enumerate}
\item[\itshape{Case 1.}] $W$ is a $\bb{R}$-valued Wiener process which is defined on the
filtered probability space $(\Omega,\cal{F},\{\cal{F}_{t}\}_{0\leq t\leq T},\bb{P})$, and $\bG[\bu]$ is $d$-dimensional nonlinear mapping;
\item[\itshape{Case 2.}]$\bW$ is a $\bb{R}^d$-valued Wiener process , and $G[\bu]$ is a scalar nonlinear mapping;
\item[\itshape{Case 3.}] $\bW$ is a $\bb{R}^{l}$-valued  Wiener process, $\bG[\bu]$ 
is a $d\times l$ matrix, then $\bG\dot{\bW}$ is a $d$-dimensional multiplicative noise.
\end{enumerate}
For the sake of presentation clarity, we only consider {\itshape Case 1}, For the other two cases,  it can be shown that the same results still hold.

\commentone{Wave propagation is a fundamental physical phenomenon, and it arises from various applications in geophysics, engineering, medical science, biology, etc.} There is a large amount of literature on  numerical methods for deterministic acoustic wave equations,   we refer the reader to \cite{adjerid2011discontinuous,baccouch2012local,baker1976error,chou2014optimal,chung2006optimal,chung2009optimal,dupont19732,falk1999explicit,grote2006discontinuous,grote2009optimal,monk2005discontinuous,riviere2003discontinuous,safjan1993high,sun2021,xing2013energy,zhong2011numerical} and the references therein for a detailed account. Moreover, numerical methods for stochastic acoustic wave equations have also been intensively developed in the last few years,  see \cite{anton2016full,cohen2013trigonometric,cohen2015fully,cohen2018numerical,cui2019strong,feng2022higher,gubinelli2018renormalization,hausenblas2010weak,hong2022energy,kovacs2010finite,li2022finite,quer2006space,walsh2006numerical}.  
Similarly, the elastic wave equations are also of great importance and find applications in geoscience for modeling seismic waves and in medical science for tumor detection as well as in materials science for non-destructive testing. Although there is a large literature in numerical methods for deterministic elastic wave equations, see  \cite{gauthier1986two,igel1995anisotropic,saenger2000modeling,virieux1984sh, marfurt1984accuracy, ciarlet2002finite,reddy1973convergence} and the references therein, there is barely any work on numerical analysis of stochastic elastic wave equations in the literature, which motivates us to carry out the work of this paper.

The primary goal of this paper is to develop some semi-discrete  (in time) scheme and fully discrete finite element methods for the stochastic elastic wave equation with multiplicative noise.
The highlight of the paper is the establishment of strong norm convergence and error estimates  for both semi-discrete and fully discrete methods. To achieve this goal, we first need to establish some stability and H\"older continuity estimates for the (variational) weak solution of the stochastic  wave equations. These results will be crucially used to derive the desired error estimates for the semi-discrete scheme. We next need to establish various (energy) stability estimates for the semi-discrete numerical solution, which are necessary for deriving the desired error estimates for the fully discrete finite element methods.     

The rest of the paper is organized as follows. In Section 2, we introduce a variational weak formulation for problem \eqref{eqn:elastic-wave-equation}--\eqref{20220526_2}. The stability and H\"older continuity estimates in the $\bL^2$-, $\bH^1$-, and $\bH^2$-norm are established for the strong solution. In Section 3, we propose a semi-discrete in time numerical scheme for problem\eqref{eqn:elastic-wave-equation}--\eqref{20220526_2}. It is proved that the semi-discrete solution is energy stable. Moreover, we prove the convergence with rate $O(k)$   in the $\bL^2$-norm  and $O(k^{\frac12})$ in the $\bH^1$-norm for the displacement approximations. In Section \ref{sec:fem-discrete}, we propose a fully discrete finite element method to discretize the semi-discrete scheme in space and derive its error estimates,  which show that for the linear finite element, the $\bL^2$-norm of the error converges with  $O(h^2)$ rate and  the $\bH^1$-norm converges with $O(h)$ rate. In Section \ref{sec:numerical-tests}, we present two two-dimensional numerical experiments to test the efficiency of the proposed numerical methods and to validate the theoretical error estimate results. Finally, we conclude the paper with a short summary given in Section \ref{sec:conclustion}.

\section{Preliminaries}\label{sec:preliminaries}
Standard notations for functions and spaces are adopted in this paper. For example, $\bL^p$ denotes $(L^p(\cal{D}))^d$ for $1\leq p\leq \infty$, $(\cdot,\cdot)$ denotes the standard $\bL^2(\cal{D})$-inner product and $\bH^m(\cal{D})$ denotes the Sobolev space of order $m$. Throughout this paper, $C$ will be used to denote a generic positive constant which is independent the mesh parameters $k$ and $h$. 

\subsection{Assumptions}\label{ssec:assumptions}
The following structural conditions will be imposed on the mappings $\bF[\cdot]$ and $\bG[\cdot]$:
\begin{align}
\|\bF[0]\|_{\bL^2}+\|\bG[0]\|_{\bL^2}\leq& C_A,\label{20220528_1}\\
\|\nabla_{\bu}\bF[\cdot]\|_{\bL^\infty}+\|\nabla_{\bu} \bG[\cdot]\|_{\bL^\infty}\leq& C_A,\label{20220528_2}\\
\|\bF_{\bu_i\bu_j}[\cdot]\|_{\bL^\infty}\leq& C_A,\quad 1\le i,j\le d,\label{20220624_15}\\
\normbfhalf{\bF[\bv]-\bF[\bw]}\leq& C_B\Bigl(\lambda\|\mathrm{div}(\bv-\bw)\|_{{\bf L}^2}^2+\mu\|\epsilon(\bv-\bw)\|_{{\bf L}^2}^2+\normbfhalf{\bv-\bw}^2\Bigr)^\half,\label{20220528_3}\\
\normbfhalf{\bG[\bv]-\bG[\bw]}\leq&C_B\Bigl(\lambda\|\mathrm{div}(\bv-\bw)\|_{{\bf L}^2}^2+\mu\|\epsilon(\bv-\bw)\|_{{\bf L}^2}^2+\normbfhalf{\bv-\bw}^2\Bigr)^\half,\label{20220528_4}
\end{align}
where $\bF_{\bu_i\bu_j}[\cdot]$ denotes the second derivative of $\bF$ with respect to $\bu_i, \bu_j$ , and $C_A$ and $C_B$ are two positive constants.

\subsection{Variational weak formulation and properties of weak solutions}\label{ssec:weakform}
In this subsection, we first give the definition of variational weak formulation and weak solutions for problem \eqref{eqn:elastic-wave-equation}--\eqref{20220526_2}. We then establish several technical lemmas that will be used in the subsequent sections.
	
Equations \eqref{eqn:elastic-wave-equation}--\eqref{20220526_2} can be written as
\begin{align}\label{eqn:elastic-wave-equation-uv}
\td\bu&=\bv\td t,\\
\td\bv&=(\opL\bu+\bF[\bu])\commentone{\td t}+\bG[\bu]\td W(t),\quad \opL\bu := \Div \sigma(\bu), 
\label{20220526_8}\\
\bu(0)&=\bu_{0},\quad\bv(0)=\bv_{0},\label{20220526_9}\\
\bu(t,\cdot)&=0.\label{20220526_10}
\end{align}

\begin{definition}\label{def:weakform-solution}
The weak formulation for problem \eqref{eqn:elastic-wave-equation-uv}--\eqref{20220526_10} 
is defined as seeking $(\bu, \bv)\in\bL^2\bigl(\Omega;\mbf{C}([0, T]; \commentone{\bL^2})\cap\bL^2((0, T),\mbf{H}^1_0)\bigr)\times\bL^2\bigl(\Omega;\mbf{C}([0, T], \bL^2)\bigr)$ such that 
\begin{alignat}{2}
\big(\bu(t), \bphi\big)=&\int_{0}^{t}\big(\bv(s), \bphi\big)\td s + (\bu_0, \bphi)&&\forall \bphi \in \bL^2,\label{eqn:weakform-1}\\
\big(\bv(t), \bpsi\big)=&-\int_{0}^{t}\lambda\big(\Div(\bu(s)), \Div(\bpsi)\big)\td s-\int_{0}^{t}\mu\big(\veps(\bu(s)), \veps(\bpsi)\big)\td s\label{eqn:weakform-2}\\
+\int_{0}^{t}&\big(\bF[\bu(s)], \bpsi\big)\td s+\int_{0}^{t}\big(\bG[\bu(s)]\td W(s), \bpsi\big) + (\bv_0, \bpsi)&&\forall \bpsi \in \bH^1_0 \nonumber
\end{alignat}
for all $(\phi, \psi)\in\bf{L}^2\times\bf{H}^1_0$. Such a pair $(\bu,\bv)$, if it exists,  is called
a (variational) weak solution to problem  \eqref{eqn:elastic-wave-equation-uv}--\eqref{20220526_10}. Moreover, if a weak solution $(\bu, \bv)$ belongs to $\bL^2\big(\Omega; \mbf{C}([0, T];$ $\bH^2\cap\bH^1_0)\big)\times\bL^2\big(\Omega; \mbf{C}([0, T];\bH^1_0)\big)$, then $(\bu, \bv)$ is called a strong solution to problem  \eqref{eqn:elastic-wave-equation-uv}--\eqref{20220526_10}.
\end{definition}

\begin{remark}
	The well-posedness of problem \eqref{eqn:elastic-wave-equation-uv}--\eqref{20220526_10} can be proved using the same technique (i.e., the Galerkin method) as done in \cite{Chow} for the acoustic stochastic wave equation with multiplicative noise. The only markable difference 
	is that \commentone{to verify the coercivity (or ellipticity) in $\bH^1(\cal{D})$ of the operator $\opL$, we need to use the well-known Korn's (second)  inequality. }
\end{remark}
	
We now state and prove the stability properties of the strong solution $(\bu, \bv)$ of problem  \eqref{eqn:elastic-wave-equation-uv}--\eqref{20220526_10}. Those bounds will be used to prove H\"older continuity in time in this section. They are also useful in establishing rates of convergence for the numerical schemes.
		
\begin{lemma}\label{lem:stability}
Let $(\bu, \bv)$ be a strong solution to equations \eqref{eqn:elastic-wave-equation-uv}--\eqref{20220526_10}. Under the assumptions \eqref{20220528_1}--\eqref{20220528_4},
 there hold
\begin{align}
\sup\limits_{0\leq t\leq T}\exc{\normbf{\bv}}+\sup\limits_{0\leq t\leq T}\exc{\normdel{\bu}}\leq& C_{s1},\label{eqn:stability-uh1}\\
\sup\limits_{0\leq t\leq T}\exc{\normdel{\bv}}+\sup\limits_{0\leq t\leq T}\exc{\normbf{\opL\bu}}\leq& C_{s2},\label{eqn:stability-vh1}\\
\sup\limits_{0\leq t\leq T}\exc{\normbf{\partial_{x_j}\bv}}+\sup\limits_{0\leq t\leq T}\exc{\normdel{\partial_{x_j}\bu}}\leq& C_{s3}\label{20220624_7}
\end{align}
for $1\leq j\leq d$, and
\begin{align*}
C_{s1}&=\Big(\exc{\normbf{\bv_0}}+\exc{\normdel{\bu_0}}+4C_A^2\Big)e^{CC_B^2},\\
C_{s2}&=\Big(\exc{\normdel{\bv_0}}+\exc{\normbf{\opL\bu_0}}+CC_A^2C_{s1}\Big)e^T,\\
C_{s3}&=\Big(\exc{\normbf{\partial_{x_j}\bv_0}}+\exc{\normdel{\partial_{x_j}\bu_0}}\Big)e^{CC_A^2}.
\end{align*}
\end{lemma}
	 
\begin{proof}\label{prf:stability}
{\itshape Step 1:} Applying It\^o's formula to the functional $\bPhi_1(\bv(\cdot))=\normbf{\bv(\cdot)}$ yields
\begin{align}\label{eqn:prf-stability-uh1-1}
&\normbf{\bv(t)}=\normbf{\bv_0}+\int_{0}^{t}\tD\bPhi_1\big(\bv(s)\big)\big(\opL\bu+\bF[\bu]\big)\td s\\
&\qquad+\int_{0}^{t}\text{Tr}\Big(\tD^2\bPhi_1\big(\bv(s)\big)\big(\bG[\bu],\bG[\bu]\big)\Big)\td s + \int_{0}^{t}\tD\bPhi_1\big(\bv(s)\big)\big(\bG[\bu]\td W(s)\big).\notag
\end{align}

The expressions of $\tD\bPhi_1(\bv(\cdot))$ and $\tD^{2}\bPhi_1(\bv(\cdot))$ are as follows:
\begin{alignat}{2}
\tD\bPhi_1(\bv)(\bw_1)&=2(\bv, \bw_1) &&\qquad \forall\bw_1\in C^{\infty}_0,\\
\tD^2\bPhi_1(\bv)(\bw_1, \bw_2)&=2(\bw_1, \bw_2) &&\qquad \forall\bw_1,\bw_2\in C^{\infty}_0.\notag
\end{alignat}

Substituting the expressions of $\tD\bPhi_1(\bv(\cdot))$ and $\tD^{2}\bPhi_1(\bv(\cdot))$ into \eqref{eqn:prf-stability-uh1-1} , we get
\begin{align}\label{eqn:prf-stability-uh1-4}
&\normbf{\bv(t)}+\normdelbig{\bu(t)}\\
&\qquad=\normbf{\bv_0}+\normdel{\bu_0}\notag\\
&\qquad\quad+2\int_{0}^{t}\big(\bF[\bu], \bv\big)\td s+\int_{0}^{t}\normbf{\bG[\bu]}\td s+2\int_{0}^{t}\big(\bG[\bu], \bv\big)\td W(s)\notag\\
&\qquad:= {\normbf{\bv_0}}+{\normdel{\bu_0}}+ I_1 + I_2 + I_3.\notag
\end{align}
		
For the first term $I_1$, let $\bar\bu=0$ in equation \eqref{20220528_3}. Using equation \eqref{20220528_1}, the Poincar\'{e} inequality, and the Korn's inequality, we have
\begin{align}\label{eqn:prf-stability-uh1-5}
2\int_{0}^{t}\big(\bF[\bu], \bv\big)\td s\leq&\int_{0}^{t}\|\bF[\bu]\|_{{\bf L}^2}^2+\normbf{\bv}\td s\\
\leq&2C_B^2\int_{0}^{t}\normdel{\bu}+\normbf{\bu}\td s+2C_A^2+\int_{0}^{t}\big(\normbf{\bv}\big)\td s\notag\\
\leq&CC_B^2\int_{0}^{t}\normdel{\bu}\td s+2C_A^2+\int_{0}^{t}\normbf{\bv}\td s\notag.
\end{align}

Taking the expectation on both sides of \eqref{eqn:prf-stability-uh1-5}, we obtain
\begin{align}\label{20220530_3}
2\mathbb{E}\Bigl[\int_{0}^{t}\commentone{\big(\bF[\bu], \bv\big)}\td s\Bigr]\leq&CC_B^2\mathbb{E}\Bigl[\int_{0}^{t}\normdel{\bu}\td s\Bigr]\\
&+2C_A^2+\mathbb{E}\Bigl[\int_{0}^{t}\normbf{\bv}\td s\Bigr].\notag
\end{align}

Similarly, using equations \eqref{20220528_1} and \eqref{20220528_4}, the Poincar\'{e} inequality, and the Korn's inequality, the expectation of the second term $I_2$ can be bounded by
\begin{align}
\mathbb{E}\Bigl[\int_{0}^{t}\normbf{\bG[\bu]}\td s\Bigr]\leq&  2C_B^2\mathbb{E}\Bigl[\int_{0}^{t}\normdel{\bu}+\normbf{\bu}\td s\Bigr]+2C_A^2\label{eqn:prf-stability-uh1-6}\\
\leq&CC_B^2\mathbb{E}\Bigl[\int_{0}^{t}\normdel{\bu}\td s\Bigr]+2C_A^2.\notag
\end{align}
			
The third term $I_3$ is a martingale, and $\exc{I_3}=0$. Taking the expectation on both sides of \eqref{eqn:prf-stability-uh1-4} and using the Gronwall's inequality, we get
\begin{align}\label{eqn:prf-stability-uh1-7}
&\exc{\normbf{\bv}}+\exc{\normdel{\bu}}\\
&\qquad\leq\Big(\exc{\normbf{\bv_0}}+\exc{\normdel{\bu_0}}+4C_A^2\Big)e^{CC_B^2}.\notag
\end{align}
		
{\itshape Step 2:} Again, by applying It\^o's formula to $\bPhi_2\big(\bu(\cdot)\big)=\normbf{\opL\bu(\cdot)}$, we obtain
\begin{equation}\label{eqn:prf-stability-uh2-1}
\normbf{\opL\bu(t)}=\normbf{\opL\bu_0}+2\int_{0}^{t}\big(\opL\bu(s), \opL\bv(s)\big)\td s.
\end{equation}

Applying It\^o's formula to $\bPhi_3\big(\bv(\cdot)\big)=\normdelbig{\bv(\cdot)}$ leads to
\begin{align}\label{eqn:prf-stability-uh2-2}
&\normdel{\bv(t)}=\normdel{\bv_0}\\
&\qquad+2\int_{0}^{t}\Big(\innerprd{\opL\bu}{\bv}\Big)\td s\notag\\
&\qquad+2\int_{0}^{t}\Big(\innerprd{\bF[\bu]}{\bv}\Big)\td s\notag\\
&\qquad+\int_{0}^{t}\big(\normdel{\bG[\bu]}\big)\td s\notag\\
&\qquad+2\int_{0}^{t}\Big(\innerprd{\bG[\bu]\td W(s)}{\bv}\Big)\td s.\notag
\end{align}

Adding \eqref{eqn:prf-stability-uh2-1} and \eqref{eqn:prf-stability-uh2-2} gives
\begin{align}\label{eqn:prf-stability-uh2-3}
&\normbf{\opL\bu(t)}+\normdel{\bv(t)}=\normdel{\bv_0}\\
&\qquad+\normbf{\opL\bu_0}+2\int_{0}^{t}\Big(\innerprd{\bF[\bu]}{\bv}\Big)\td s\notag\\
&\qquad+\int_{0}^{t}\big(\normdel{\bG[\bu]}\big)\td s\notag\\
&\qquad+2\int_{0}^{t}\Big(\innerprd{\bG[\bu]\td W(s)}{\bv}\Big)\td s\notag\\
&:=\normdel{\bv_0}+\normbf{\opL\bu_0}+I_1 + I_2 + I_3.\notag
\end{align}
		
By equation \eqref{20220528_2} and the Korn's inequality, the first term $I_1$ can be bounded by
\begin{align}\label{eqn:prf-stability-uh2-4}
&2\int_{0}^{t}\Big(\innerprd{\bF[\bu]}{\bv}\Big)\td s\\
&\qquad\leq \int_{0}^{t}\big(\normdel{\bF[\bu]}\big)\td s+\int_{0}^{t}\big(\normdel{\bv}\big)\td s\notag\\
&\qquad\leq {CC_A^2}\int_{0}^{t}\big(\lambda\normbf{\Div(\bu)}+\mu\normbf{\veps(\bu)}\big)\td s+\int_{0}^{t}\big(\normdel{\bv}\big)\td s\notag.
\end{align}

Similarly, by equation \eqref{20220528_2} and the Korn's inequality, the second term $I_2$ can be bounded by
\begin{align}
&\int_{0}^{t}\big(\normdel{\bG[\bu]}\big)\td s\le {CC_A^2}\int_{0}^{t}\big(\lambda\normbf{\Div(\bu)}+\mu\normbf{\veps(\bu)}\big)\td s.\label{eqn:prf-stability-uh2-5}
\end{align}
			
The third term $I_3$ is a martingale, and $\exc{I_3}=0$. Using equation \eqref{eqn:stability-uh1}, taking the expectation on both sides of \eqref{eqn:prf-stability-uh2-3} and using the Gronwall's inequality, we have
\begin{align}\label{eqn:prf-stability-uh2-6}
&\exc{\normbf{\opL\bu}}+\exc{\normdel{\bv}}\\
&\qquad\leq\big(\exc{\normdel{\bv_0}}+\exc{\normbf{\opL\bu_0}}+CC_A^2C_{s1}\big)e^T.\notag
\end{align}

{\itshape Step 3:} Similar to Step 1, by applying It\^o's formula to the functional $\normbf{\partial_{x_j}\bv(\cdot)}$ for $1\leq j\leq d$, we get
\begin{align}\label{20220624_8}
&\normbf{\partial_{x_j}\bv(t)}+\normdelbig{\partial_{x_j}\bu(t)}\\
&\quad=\normbf{\partial_{x_j}\bv_0}+\normdel{\partial_{x_j}\bu_0}+\int_{0}^{t}\normbf{\partial_{x_j}\bG[\bu]}\td s\notag\\
&\quad\quad+2\int_{0}^{t}\big(\partial_{x_j}\bF[\bu], \partial_{x_j}\bv\big)\td s+2\int_{0}^{t}\big(\partial_{x_j}\bG[\bu], \partial_{x_j}\bv\big)\td W(s)\notag\\
&\quad:= {\normbf{\partial_{x_j}\bv_0}}+{\normdel{\partial_{x_j}\bu_0}}+ I_1 + I_2 + I_3.\notag
\end{align}

Using equation \eqref{20220528_2} and Korn's inequality, we bound $I_1$ and $I_2$ by
\begin{align}\label{20220624_9}
2\mathbb{E}\Bigl[\int_{0}^{t}\big(\partial_{x_j}\bF[\bu], \partial_{x_j}\bv\big)\td s\Bigr]&\leq CC_A^2\mathbb{E}\Bigl[\int_{0}^{t}\lambda\bigl\|\mathrm{div}(\partial_{x_j}\bu)\bigr\|_{{\bf L}^2}^2\\
&\qquad+\mu\bigl\|\epsilon(\partial_{x_j}\bu)\bigr\|_{{\bf L}^2}^2\td s\Bigr]+\mathbb{E}\Bigl[\int_{0}^{t}\normbf{\partial_{x_j}\bv}\td s\Bigr],\notag\\
\mathbb{E}\Bigl[\int_{0}^{t}\normbf{\partial_{x_j}\bG[\bu]}\td s\Bigr]&\leq CC_A^2\mathbb{E}\Bigl[\int_{0}^{t}\lambda\bigl\|\mathrm{div}(\partial_{x_j}\bu)\bigr\|_{{\bf L}^2}^2+\mu\bigl\|\epsilon(\partial_{x_j}\bu)\bigr\|_{{\bf L}^2}^2\td s\Bigr]\label{20220624_10}.
\end{align}

The third term $I_3$ is a martingale, and $\exc{I_3}=0$. Taking the expectation on both sides of \eqref{eqn:prf-stability-uh1-4} and using Gronwall's inequality, we get
\begin{align}\label{20220624_11}
&\exc{\normbf{\partial_{x_j}\bv}}+\exc{\normdel{\partial_{x_j}\bu}}\\
&\qquad\leq\Big(\exc{\normbf{\partial_{x_j}\bv_0}}+\exc{\normdel{\partial_{x_j}\bu_0}}\Big)e^{CC_A^2}.\notag
\end{align}

The proof of Lemma \ref{lem:stability} is complete.
\end{proof}

The following lemma is also needed to establish another stability property of the strong solution $(\bu, \bv)$, which will be given in Lemma \ref{lem20220624_2}.

\begin{lemma}\label{lem20220624_1}
Let $(\bu, \bv)$ be a strong solution to equations \eqref{eqn:elastic-wave-equation-uv}--\eqref{20220526_10}. Under the assumptions \eqref{20220528_1}--\eqref{20220528_4}, there hold
\begin{align}\label{20220624_12}
\mathbb{E}\bigl[\|\opL\bF[\bu]\|_{\mathbf{L}^2}^2\bigr]&\leq CC_A^2C_{s1}+CC_A^2C_{s3},\\
\mathbb{E}\bigl[\|\opL\bG[\bu]\|_{\mathbf{L}^2}^2\bigr]&\leq CC_A^2C_{s1}+CC_A^2C_{s3}.\label{20220625_3}
\end{align}
\end{lemma}
\begin{proof}
Notice that 
\begin{equation}\label{20220625_1}
\opL\bF[\bu]=\mathrm{div}\bigl((\lambda \Div\bF[\bu]){\bf{I}}\bigr)+\mathrm{div}\bigl(\mu\veps(\bF[\bu])\bigr):=\commentone{I_1+I_2.}
\end{equation}

The first term $I_1$ on the right-hand side of \eqref{20220625_1} can be written as
\begin{align}
I_1 &= \mathrm{div}\bigl((\lambda (\nabla_{\bu}\bF)^T:\nabla\bu)\bf{I}\bigr)\\
& = \lambda\nabla\bigl((\nabla_{\bu}\bF)^T:\nabla\bu\bigr)\notag\\
& = \lambda\sum\limits_{j=1}^d\bigl[(\nabla((\nabla_{\bu}\bF)^T)_j)^T(\nabla\bu)_j+(\nabla(\nabla\bu)_j)^T((\nabla_{\bu}\bF)^T)_j\bigr],\notag
\end{align}
where $(\nabla_{\bu}\bF)^T:\nabla\bu$ denotes the element-wise product of two matrices, and $((\nabla_{\bu}\bF)^T)_j$ and $(\nabla\bu)_j$ denote the $j$-th columns of $(\nabla_{\bu}\bF)^T$ and $\nabla\bu$, respectively.

By equations \eqref{20220528_2}--\eqref{20220624_15}, equations \eqref{eqn:stability-uh1}--\eqref{20220624_7}, and the Korn's inequality, we have
\begin{align}\label{20220625_4}
\mathbb{E}\bigl[\|I_1\|_{{\bf L}^2}^2\bigr]&\le CC_A^2C_{s1}+CC_A^2C_{s3}.
\end{align}

Define $A$ and $B$ by their $ij$-th components as below:
\begin{alignat}{2}
A_{ij}&=\frac12\nabla_{\bu}\bF^j\cdot\bu_{x_i}\,\qquad 1\le i,j\le d,\\
B_{ij}&=\frac12\nabla_{\bu}\bF^i\cdot\bu_{x_j}\,\qquad 1\le i,j\le d,
\end{alignat}
where $\bF^i$ and $\bF^j$ denote the $i$-th and $j$-th components of $\bF$, respectively.

Then the second term $I_2$ on the right-hand side of \eqref{20220625_1} can be written as
\begin{equation}
I_2 = \mu\,\mathrm{div}(A+B).
\end{equation}

By equations \eqref{20220528_2}--\eqref{20220624_15}, equations \eqref{eqn:stability-uh1}--\eqref{20220624_7}, and the Korn's inequality, we have
\begin{align}\label{20220625_5}
\mathbb{E}\bigl[\|I_2\|_{{\bf L}^2}^2\bigr]&\le C\mathbb{E}\Bigl[\sum\limits_{i=1}^d\sum\limits_{j=1}^d\sum\limits_{k=1}^d\|(\nabla_{\bu}\bF^i\cdot\bu_{x_j})_{x_k}\|_{{\bf L}^2}^2\Bigr]\\
&\le CC_A^2C_{s1}+CC_A^2C_{s3}.\notag
\end{align}

Combining equations \eqref{20220625_4} and \eqref{20220625_5}, the conclusion \eqref{20220624_12} is proved. The conclusion \eqref{20220625_3} can be similarly proved.
\end{proof}

\begin{lemma}\label{lem20220624_2}
Let $(\bu, \bv)$ be a strong solution to equations \eqref{eqn:elastic-wave-equation-uv}--\eqref{20220526_10}. Under the assumptions \eqref{20220528_1}--\eqref{20220528_4}, there holds
\begin{align}
\sup\limits_{0\leq t\leq T}\exc{\normbf{\opL\bv}}+\sup\limits_{0\leq t\leq T}\exc{\normdel{\opL\bu}}\leq& C_{s4},\label{eqn:stability-vh2}
\end{align}
where
\begin{align*}
C_{s4}&=\Big(\exc{\normbf{\opL\bv_0}} + \bb{E}[\lambda\normbf{\mathrm{div}(\opL\bu_0)}+\mu\normbf{\epsilon(\opL\bu_0)}]+CC_A^2(C_{s1}+C_{s3})\big)e^{T}.
\end{align*}
\end{lemma}

\begin{proof}
Applying It\^o's formula to $\commentone{\bPhi_4\big(\bu(\cdot),\bv(\cdot)\big)}=\normdelbig{\opL\bu(\cdot)}+\normbf{\opL\bv(\cdot)}$, we obtain
\begin{align}\label{eqn:prf-stability-vh2-3}
&\normdelbig{\opL\bu(t)}+\normbf{\opL\bv(t)}\\
&\qquad =\normdel{\opL\bu_0}+\normbf{\opL\bv_0}+2\int_{0}^{t}\big(\opL\bF[\bu], \opL\bv\big)\td s\notag\\
&\qquad\quad +\int_{0}^{t}\normbf{\opL\bG[\bu]}\td s+2\int_{0}^{t}\big(\opL\bG[\bu]\td W(s), \opL\bv\big)\notag\\
&\qquad:= \normdel{\opL\bu_0}+\normbf{\opL\bv_0}+ I_1 + I_2 + I_3.\notag
\end{align}
 
By equation \eqref{20220624_12}, the expectation of the first term $I_1$ can be bounded by
\begin{align}\label{eqn:prf-stability-vh2-4}
\mathbb{E}\Bigl[2\int_{0}^{t}\big(\opL\bF[\bu], \opL\bv\big)\td s\Bigr] 
&\le\int_{0}^{t}\mathbb{E}\bigl[\|\opL\bF[\bu]\|_{\mathbf{L}^2}^2\bigr]ds + \int_{0}^{t}\mathbb{E}\bigl[\normbf{\opL\bv}\bigr]\td s   \\
&\leq CC_A^2C_{s1}+CC_A^2C_{s3}+ \int_{0}^{t}\mathbb{E}\bigl[\normbf{\opL\bv}\bigr]\td s\notag.
\end{align}

By equation \eqref{20220625_3}, the expectation of the second term $I_2$ can be bounded by
\begin{align}\label{eqn:prf-stability-vh2-5}
\mathbb{E}\Bigl[\int_{0}^{t}\normbf{\opL\bG[\bu]}\td s\Bigr]\le CC_A^2C_{s1}+CC_A^2C_{s3}.
\end{align}
		
The third term $I_3$ is a martingale, and $\exc{I_3}=0$. Taking the exceptation on the both side of \eqref{eqn:prf-stability-vh2-3} and using Gronwall's inequality, we have
\begin{align}\label{eqn:prf-stability-vh2-6}
&\exc{\normdel{\opL\bu}}+\exc{\normbf{\opL\bv}}\\
&\qquad \leq\big(\exc{\normbf{\opL\bv_0}} + \bb{E}[\lambda\normbf{\mathrm{div}(\opL\bu_0)}+\mu\normbf{\epsilon(\opL\bu_0)}]+CC_A^2(C_{s1}+C_{s3})\big)e^{T}.\notag
\end{align}
	
The proof of Lemma \ref{lem:stability} is complete.
\end{proof}

Based on Lemma \ref{lem:stability}, we can establish H\"older continuity in time of the solution $\bu$ in $\bf{L}^2$-norm, $\bf{H}^1$-seminorm, and $\bf{H}^2$-seminorm. These results play a key role in the error analysis.
	
\begin{lemma}\label{lem:holder-cont-u}
{\itshape
Let $(\bu, \bv)$ be a strong solution to problem \eqref{eqn:elastic-wave-equation-uv}--\eqref{20220526_10}. Under the assumptions \eqref{20220528_1}--\eqref{20220528_4}, for any $s, t\in[0, T]$, we have
\begin{align}
\exc{\normbfp{\bu(t)-\bu(s)}}\leq& C_{s1}|t-s|^{2p},\label{eqn:holder-cont-u-1}\\
\exc{\normdelpbig{\bu(t)-\bu(s)}}\leq& C_{s2}|t-s|^{2p},\label{eqn:holder-cont-u-2}\\
\exc{\normbf{\opL\big(\bu(t)-\bu(s)\big)}}\leq& C_{s4}|t-s|^2.\label{eqn:holder-cont-u-3}
\end{align}	
}
\end{lemma}

\begin{proof}\label{prf:holder-cont-u}
Note that $\bu(t)-\bu(s)=\int_{s}^{t}\bv(\xi)\td\xi$, and then we get
\begin{align}\label{eqn:prf-holder-cont-u-1}
\exc{\normbfp{\bu(t)-\bu(s)}}&\leq\exc{\normbfp{\left(\int_{s}^{t}|\bv(\xi)|^2\td\xi\right)^{\half}\left(\int_{s}^{t}1\td\xi\right)^{\half}}}\\
&\leq |t-s|^p\exc{\left(\int_{s}^{t}\normbf{\bv(\xi)}\td\xi\right)^p}\notag\\
&\leq \sup\limits_{0\leq t\leq T}\exc{\normbf{\bv(t)}}|t-s|^{2p}.\notag
\end{align}
		
Hence, \eqref{eqn:holder-cont-u-1} holds when applying \eqref{eqn:stability-uh1} in Lemma \ref{lem:stability}. The inequality \eqref{eqn:holder-cont-u-2} and \eqref{eqn:holder-cont-u-3} could be derived similarly.
\end{proof}

The next Lemma establishes some H\"older continuity in time results for the solution $\bv$ with respect to $\bf{L}^2$-norm, $\bf{H}^1$-seminorm and $\bf{H}^2$-seminorm.
	
\begin{lemma}\label{lem:holder-cont-v}
{\itshape
Let $(\bu, \bv)$ be a strong solution to problem \eqref{eqn:elastic-wave-equation-uv}--\eqref{20220526_10}. Under the assumptions \eqref{20220528_1}--\eqref{20220528_4}, for any $s, t\in[0, T]$, we have
\begin{align}
\exc{\normbf{\bv(t)-\bv(s)}}\leq& C_{s5}|t-s|,\label{eqn:holder-cont-v-1}\\
\exc{\normdelbig{\bv(t)-\bv(s)}}\leq& C_{s6}|t-s|,\label{eqn:holder-cont-v-2}\\
\exc{\normbf{\opL\big(\bv(t)-\bv(s)\big)}}\leq& C_{s7}|t-s|,\label{eqn:holder-cont-v-3}
\end{align}
where
\begin{align*}
C_{s5}&=e^T(CC_B^2C_{s1}+4C_A^2),\\
C_{s6}&=e^T(CC_A^2C_{s1}+CC_B^2C_{s1}+2C_A^2),\\
C_{s7}&=CC_A^2(C_{s1}+C_{s2}+C_{s4}).
\end{align*}
}
\end{lemma}

\begin{proof}\label{prf:holder-cont-v}
{\itshape Step 1:} Fix $s\geq0$. Applying It\^o's formula to $\bPhi_5(\bu(\cdot)):=\lambda\|\mathrm{div}\bigl(\bu(\cdot)-\bu(s)\bigr)\|_{{\bf L}^2}^2+\mu\|\epsilon\bigl(\bu(\cdot)-\bu(s)\bigr)\|_{{\bf L}^2}^2$, we have
\begin{align}\label{eqn:prf-holder-cont-vl2-1}
&\normdelbig{\bu(t)-\bu(s)}\\
&\qquad =2\int_{s}^{t}\Big(\innerprd{\bu(\xi)-\bu(s)}{\bv(\xi)}\Big)\td\xi.\notag
\end{align}
		
Applying It\^o's formula to $\bPhi_6(\bv(\cdot)):=\normbf{\bv(\cdot)-\bv(s)}$ and using \eqref{eqn:prf-holder-cont-vl2-1}, we get
\begin{align}\label{eqn:prf-holder-cont-vl2-2}
&\normbf{\bv(t)-\bv(s)}+\normdelbig{\bu(t)-\bu(s)}\\
&\qquad= 2\int_{s}^{t}\big(\bF[\bu(\xi)], \bv(\xi)-\bv(s)\big)\td\xi + \int_{s}^{t}\normbf{\bG[\bu(\xi)]}\td\xi\notag\\
&\qquad\quad+2\int_{s}^{t}\big(\bG[\bu(\xi)], \bv(\xi)-\bv(s)\big)\td W(\xi)\notag\\
&\qquad:= I_1 + I_2 + I_3.\notag
\end{align}
		
Similar to the analysis in \eqref{eqn:prf-stability-uh1-5}--\eqref{eqn:prf-stability-uh1-6}, and by equation \eqref{eqn:stability-uh1}, we have
\begin{align}\label{20220603_1}
&\int_{s}^{t}\normbf{\bF[\bu(\xi)]}\td\xi+\int_{s}^{t}\normbf{\bG[\bu(\xi)]}\td\xi\\
&\qquad \le CC_B^2\int_s^t\big(\normdel{\bu}\big)\td\xi+4C_A^2|t-s|\notag\\
&\qquad\le(CC_B^2C_{s1}+4C_A^2)|t-s|\notag.
\end{align}

Then the first term and the second term on the right-hand side of \eqref{eqn:prf-holder-cont-vl2-2} can be bounded by
\begin{align}\label{eqn:prf-holder-cont-vl2-3}
&2\int_{s}^{t}\big(\bF[\bu(\xi)], \bv(\xi)-\bv(s)\big)\td\xi + \int_{s}^{t}\normbf{\bG[\bu(\xi)]}\td\xi\\
&\qquad\leq\int_{s}^{t}\normbf{\bv(\xi)-\bv(s)}\td\xi+\int_{s}^{t}\big(\normbf{\bF[\bu(\xi)]}+\normbf{\bG[\bu(\xi)]}\big)\td\xi\notag\\
&\qquad\leq\int_{s}^{t}\normbf{\bv(\xi)-\bv(s)}\td\xi+(CC_B^2C_{s1}+4C_A^2)|t-s|.\notag
\end{align}

The third term is a martingale and $\exc{\int_{s}^{t}\big(\bG[\bu(\xi)], \bv(\xi)-\bv(s)}=0$. Taking the expectation on both sides of \eqref{eqn:prf-holder-cont-vl2-2}, we get
\begin{align}\label{eqn:prf-holder-cont-vl2-4}
&\exc{\normbf{\bv(t)-\bv(s)}}+\exc{\normdelbig{\bu(t)-\bu(s)}}\\
&\qquad\leq\int_{s}^{t}\exc{\normbf{\bv(\xi)-\bv(s)}}d\xi+(CC_B^2C_{s1}+4C_A^2)|t-s|.\notag
\end{align}

Using Gronwall's inequality yields
\begin{align}\label{eqn:prf-holder-cont-vl2-5}
&\exc{\normbf{\bv(t)-\bv(s)}}+\exc{\normdelbig{\bu(t)-\bu(s)}}\\
&\qquad\leq e^T(CC_B^2C_{s1}+4C_A^2)|t-s|.\notag
\end{align}
		
{\itshape Step 2:} Fix $s\geq 0$. Applying It\^o's formula to $\bPhi_7\big(\bu(\cdot)\big)=\normbf{\opL\bu(\cdot)-\opL\bu(s)}$ and using integration by parts yield
\begin{equation}\label{eqn:prf-holder-cont-vh1-1}
\normbf{\opL\bu(t)-\opL\bu(s)}=2\int_{s}^{t}\big(\opL\bu(\xi)-\opL\bu(s), \opL\bv(\xi)\big)\td\xi.
\end{equation}
	
Applying It\^o's formula to $\bPhi_8\big(\bv(\cdot)\big)=\normdelbig{\bv(\cdot)-\bv(s)}$ and using \eqref{eqn:prf-holder-cont-vh1-1} yield
\begin{align}\label{eqn:prf-holder-cont-vh1-2}
&\normbf{\opL\bu(t)-\opL\bu(s)}+\normdelbig{\bv(\cdot)-\bv(s)}\\
&\quad \leq2\int_{s}^{t}\Bigl(\lambda\big(\Div(\bv(\xi)-\bv(s)), \Div(\bF[\bu(\xi)])\big)\Bigr.\notag\\
&\qquad\quad + \Bigl.\mu\big(\veps(\bv(\xi)-\bv(s)), \veps(\bF[\bu(\xi)])\big)\Bigr)\td\xi+ \int_{s}^{t}\normbf{\bG[\bu(\xi)]}\td\xi\notag\\
&\qquad\quad  +2\int_{s}^{t}\lambda\big(\Div\big(\bv(\xi)-\bv(s)\big)\td W(\xi),\Div(\bG[\bu(\xi)])\big)\td\xi\notag\\
&\qquad\quad + 2\int_{s}^{t}\mu\big(\veps\big(\bv(\xi)-\bv(s)\big)\td W(\xi),\veps(\bG[\bu(\xi)])\big)\td\xi\notag\\
&\quad :=I_1+I_2+I_3+I_4.\notag
\end{align}
	
By equations \eqref{20220528_2} and \eqref{eqn:stability-uh1}, and the Korn's inequality, the first term $I_1$ can be bounded by
\begin{align}\label{eqn:prf-holder-cont-vh1-3}
&2\int_{s}^{t}\Big(\innerprd{\bv(\xi)-\bv(s)}{\bF[\bu(\xi)]}\Big)\td\xi\\
&\qquad\leq\int_{s}^{t}\lambda\normbf{\Div\big(\bF[\bu(\xi)]\big)}\td\xi+\int_{s}^{t}\mu\normbf{\veps\big(\bF[\bu(\xi)]\big)}\td\xi\notag\\
&\qquad\quad +\int_{s}^{t}\big(\normdelbig{\bv(\xi)-\bv(s)}\big)\td\xi\notag\\
&\qquad\leq CC_A^2\int_{s}^{t}\big(\normdelbig{\bu(\xi)}\big)\td\xi\notag\\
&\qquad\quad +\int_{s}^{t}\lambda\normbf{\Div\big(\bv(\xi)-\bv(s)\big)}\td\xi+\int_{s}^{t}\mu\normbf{\veps\big(\bv(\xi)-\bv(s)\big)}\td\xi\notag\\
&\qquad\leq \int_{s}^{t}\big(\normdelbig{\bv(\xi)-\bv(s)}\big)\td\xi+CC_A^2C_{s1}|t-s|.\notag
\end{align}

Similar to equation \eqref{eqn:prf-stability-uh1-6}, by  Poincar\'{e} inequality,  Korn's inequality, and equations \eqref{20220528_1}, \eqref{20220528_4} and \eqref{eqn:stability-uh1}, the second term $I_2$ can be bounded by
\begin{align}\label{20220610_1}
\int_{s}^{t}\normbf{\bG[\bu]}\td s\leq&CC_B^2\int_{s}^{t}\normdel{\bu}\td s+2C_A^2|t-s|\\
\le&(CC_B^2C_{s1}+2C_A^2)|t-s|.\notag
\end{align}

The last two terms $I_3$ and $I_4$ are martingales, so $\exc{I_3+I_4}=0$. Taking the expectation on both sides of equation \eqref{eqn:prf-holder-cont-vh1-2} and using \eqref{eqn:prf-holder-cont-vh1-3}-\eqref{20220610_1} give
\begin{align}\label{eqn:prf-holder-cont-vh1-5}
&\exc{\normbf{\opL\bu(t)-\opL\bu(s)}}+\exc{\normdelbig{\bv(\cdot)-\bv(s)}}\\
&\qquad\leq\int_{s}^{t}\exc{\normdelbig{\bv(\xi)-\bv(s)}}\td\xi\notag\\
&\qquad\quad+(CC_A^2C_{s1}+CC_B^2C_{s1}+2C_A^2)|t-s|.\notag
\end{align}
	
Then by Gronwall's inequality we obtain 
\begin{align}\label{eqn:prf-holder-cont-vh1-6}
&\exc{\normdelbig{\bv(t)-\bv(s)}} \\
&\qquad\leq e^T(CC_A^2C_{s1}+CC_B^2C_{s1}+2C_A^2)|t-s|.\notag
\end{align}
	
{\itshape Step 3:} Fix $s\geq 0$. Applying It\^o's formula to  
$$\commentone{\bPhi_9\big(\bu(\cdot)\big)}=\normdelbig{\opL\bu(\cdot)-\opL\bu(s)},$$ 
and then we get
\begin{align}\label{eqn:prf-holder-cont-vh2-1}
&\normdelbig{\opL\bu(t)-\opL\bu(s)}\\
&\qquad= 2\int_{s}^{t}\Big(\lambda\big(\Div(\opL\bu(\xi)-\opL\bu(s)),\Div(\opL\bv(\xi)) \big)\Big.\notag\\
&\qquad\quad+\Big.\mu\big(\veps(\opL\bu(\xi)-\opL\bu(s)),\veps(\opL\bv(\xi)) \big)\Big)\td\xi.\notag
\end{align}

Applying It\^o's formula to $\commentone{\bPhi_{10}\big(\bv(\cdot)\big)}=\normbf{\opL\bv(\cdot)-\opL\bv(s)}$, and using integration by parts and equation \eqref{eqn:prf-holder-cont-vh2-1}, we have
\begin{align}\label{eqn:prf-holder-cont-vh2-2}
&\normbf{\opL\bv(t)-\opL\bv(s)}+\normdelbig{\opL\bu(t)-\opL\bu(s)}\\
&\qquad = 2\int_{s}^{t}\big(\opL\bv(\xi)-\opL\bv(s), \opL\bF[\bu(\xi)]\big)\td\xi + \int_{s}^{t}\normbf{\opL\bG[\bu(\xi)]}\td\xi\notag\\
&\qquad\quad + 2\int_{s}^{t}\big(\opL\bv(\xi)-\opL\bv(s), \opL\bG[\bu(\xi)]\td W(\xi)\big)\notag\\
&\qquad := I_1 + I_2 + I_3.\notag
\end{align}
	
Similar to the estimation of equation \eqref{eqn:prf-stability-vh2-5}. we have
\begin{align}\label{20220610_3}
&\int_{s}^{t}\normbf{\opL\bF[\bu(\xi)]}+\normbf{\opL\bG[\bu(\xi)]}\td s\\
&\qquad \leq {CC_A^2}\int_{s}^{t}\lambda\normbf{\Div(\opL\bu)}+\mu\normbf{\veps(\opL\bu)}+\normbf{\opL\bu}\td s\notag\\
&\qquad\quad+{CC_A^2}\int_{s}^{t}\lambda\normbf{\Div(\bu)}+\mu\normbf{\veps(\bu)}\td s\notag\\
&\qquad \leq CC_A^2(C_{s1}+C_{s2}+C_{s4})|t-s|.\notag
\end{align}

By  equation \eqref{20220610_3}, the first and second term on the right-hand side of \eqref{eqn:prf-holder-cont-vh2-2} could be written as
\begin{align}\label{eqn:prf-holder-cont-vh2-3}
&2\int_{s}^{t}\big(\opL\bv(\xi)-\opL\bv(s), \opL\bF[\bu(\xi)]\big)\td\xi+\int_{s}^{t}\normbf{\opL\bG[\bu(\xi)]}\td\xi\\
&\qquad \leq \int_{s}^{t}\normbf{\opL\bF[\bu(\xi)]}+\normbf{\opL\bG[\bu(\xi)]}\td\xi\notag+ \int_{s}^{t}\normbf{\opL\bv(\xi)-\opL\bv(s)}\td\xi\notag\\
&\qquad\leq CC_A^2(C_{s1}+C_{s2}+C_{s4})|t-s|+ \int_{s}^{t}\normbf{\opL\bv(\xi)-\opL\bv(s)}\td\xi.\notag
\end{align}
		
The third term $I_3$ is a martingale and $\exc{I_3}=0$. Taking the exceptation on both sides of \eqref{eqn:prf-holder-cont-vh2-2}, we get
\begin{align}\label{eqn:prf-holder-cont-vh2-4}
&\exc{\normbf{\opL\bv(t)-\opL\bv(s)}}\\
&\qquad +\bb{E}\left[\lambda\normbf{\Div(\opL\bu(t)-\opL\bu(s))}+\mu\normbf{\veps(\opL\bu(t)-\opL\bu(s))}\right]\notag\\
&\ \leq CC_A^2(C_{s1}+C_{s2}+C_{s4})|t-s|+\int_{s}^{t}\exc{\normbf{\opL\bv(\xi)-\opL\bv(s)}}\td\xi.\notag
\end{align}
	
Then the Gronwall's inequality yields
\begin{equation}\label{eqn:prf-holder-cont-vh2-5}
\begin{aligned}
\exc{\normbf{\opL\bv(t)-\opL\bv(s)}}\leq CC_A^2(C_{s1}+C_{s2}+C_{s4})|t-s|.
\end{aligned}
\end{equation}

The proof is complete.
\end{proof}

\section{Semi-discretization in time}\label{sec:time-discrete}
In this section we propose a time semi-discrete scheme to approximate the nonlinear stochastic elastic wave equations. The goals are to prove some stability results and to establish the error estimates.

\subsection{Time semi-discrete scheme}\label{ssec:time-semiform}
Let $0<k<<T$ and $t_{n}:=nk$ for $n=0,1,2,\cdots,N$ be uniform meshes with size $k$ on the interval $[0,T]$.
	
\begin{scheme}\label{sch:scheme1}
{\itshape		
Let $\bu^0$ be an $\cal{F}_{t_0}$-measurable and $\bf{H}^1_0$-valued random variable and $\bv^0$ be an \commentone{$\cal{F}_{t_0}$-measurable} and $\bf{H}^1_0$-valued random variable. 
For each $n \ge 1$, find $(\bf{H}^1_0\times \bL^2)$-valued and $\{\cal{F}_{t_{n+1}}\}$-measurable random variables $(\bu^{n+1}, \bv^{n+1})$ such that $\bb{P}$-a.s.
\begin{alignat}{2}\label{eqn:time-semiform}
\big(\bu^{n+1}-\bu^n, \bphi\big) = & k\big(\bv^{n+1}, \bphi\big) 
&&\forall \bphi \in \bL^2,\\  
(\bv^{n+1}-\bv^n, \bpsi) = &\innerprdminusk{\bu^{n+\half}}{\bpsi} &\label{20220529_1}\\
&  + \big(\bG[\bu^n]\commentone{\Delta W_{n}}, \bpsi\big)+ k\big(\bF[\bu^n], \bpsi\big)	&&\forall \bpsi \in \bH^1_0,\notag
\end{alignat}
where \commentone{$\Delta W_{n}:=W(t_{n+1})-W(t_n)$} and $\bu^{n+\half}:= \half\big(\bu^{n+1}+\bu^{n}\big)$.
}
\end{scheme}

\begin{remark}\label{rmk:sch1'}
(a) At each time step, the above scheme is a nonlinear random PDE system for $(\bu^{n+1}, \bv^{n+1})$ whose well-posedness can be proved by a standard fixed point argument based on the stability estimates to be given in the next subsection. 
	
(b) Following \cite{dupont19732}, a possible improvement to Scheme \ref{sch:scheme1} is the following modified scheme: Seeking $(\bf{H}^1_0\times\bL^2)$-valued and $\cal{F}_{t_{n+1}}$-measurable random variables $(\bu^{n+1}, \bv^{n+1})$ such that $\bb{P}$-a.s.
\begin{alignat}{2}\label{20220529_2}
\big(\bu^{n+1}-\bu^n, \bphi\big) = & k\big(\bv^{n+1}, \bphi\big) 
&&\forall \bphi \in \bL^2,\\  
(\bv^{n+1}-\bv^n, \bpsi) = &\innerprdminusk{\bu^{n,\half}}{\bpsi} &\label{20220529_3}\\
&+\big(\bG[\bu^n]\commentone{\Delta W_{n}}, \bpsi\big)+ k\big(\bF[\bu^n], \bpsi\big)	&&\forall \bpsi \in \bH^1_0,\notag
\end{alignat}
where $\bu^{n, \half}:= \half\big(\bu^{n+1}+\bu^{n-1}\big)$.

However, such an improvement could not be realized unless some more involved higher order treatment of the noise term is adopted as demonstrated in \cite{feng2022higher} for the corresponding  stochastic acoustic wave equations. 
\end{remark}

\subsection{Stability analysis of the time semi-discrete scheme}\label{ssec:time-semiform-stability}
Definie the following energy functional:
\begin{equation}\label{eqn:eng-norm}
J^n=J(\bu^{n},\bv^{n}):=\commentone{\frac{1}{2}\Big(\|\bv^{n}\|^{2}_{\bL^{2}}+\lambda\|\Div(\bu^{n})\|^{2}_{\bL^{2}}+\mu\|\veps(\bu^{n})\|^{2}_{\bL^{2}}\Big)}.
\end{equation}

The following lemma gives the estimate for the expectation of the above energy functional.
\begin{lemma}\label{lemma:time-semiform-stab}
Let $I_k:=\{t_n\}^N_{n=0}$ be uniform meshes with size $k$ satisfying $0<k\leq k_0 < T$ and $(\bu^0, \bv^0)=(\bu_0, \bv_0)\commentone{\in(\bH^2\cap\bH^1_0)\times\bH^1_0}$ be given. Then there holds
\begin{equation}\label{eqn:time-semiform-stab}
\max_{1\leq n\leq N}\bb{E}\bigl[J(\bu^{n},\bv^{n})\bigr]\leq \bigl(\exc{J(\bu^0, \bv^0)}+4C_A^2\bigr)e^{CC_B^2}.
\end{equation}
\end{lemma}

\begin{proof}\label{prf:time-semiform-stab}
Fix $\omega\in\Omega$ and choose $\bpsi=\bv^{n+1}$ in \eqref{20220529_1}, and then we get
\begin{align}\label{eqn:prf-time-semiform-stab-1}
\half\left[\normbf{\bv^{n+1}}\right.&\left.-\normbf{\bv^n}+\normbf{\bv^{n+1}-\bv^n}\right]\\
&=k\innerlap{\bu^{n+\half}}{\bv^{n+1}}+\big(\bG[\bu^n]\commentone{\Delta W_{n}},\bv^{n+1}\big)+k\big(\bF[\bu^{n}],\bv^{n+1}\big).\notag
\end{align}
	
Choose $\bphi=-\opL\bu^{n+\half}$ in \eqref{eqn:time-semiform}, and then we get
\begin{equation}\label{eqn:prf-time-semiform-stab-2}
(\bu^{n+1}-\bu^n,-\opL\bu^{n+\half})=-k\innerlap{\bu^{n+\half}}{\bv^{n+1}}.
\end{equation}
	
Combining \eqref{eqn:prf-time-semiform-stab-1} and \eqref{eqn:prf-time-semiform-stab-2} and taking the expectation on both sides yield 
\begin{align}\label{eqn:prf-time-semiform-stab-3}
&\exc{(\bu^{n+1}-\bu^n,-\opL\bu^{n+\half})}+\half\left(\exc{\normbf{\bv^{n+1}}}-\exc{\normbf{\bv^n}}\right.\\
&\qquad\qquad \left.+\exc{\normbf{\bv^{n+1}-\bv^n}}\right)\notag\\
&\qquad = \exc{\big(\bG[\bu^n]\commentone{\Delta W_{n}},\bv^{n+1}\big)}+\exc{k\big(\bF[\bu^{n}],\bv^{n+1}\big)}:=I_1 + I_2.\notag
\end{align}

For the first term on the right-hand side, note that $\exc{\big(\bG[\bu^n]\commentone{\Delta W_{n}},\bv^n\big)}=0$. Similar to the estimation of \eqref{eqn:prf-stability-uh1-6}, and by It\^o's isometry, we obtain
\begin{align}\label{eqn:prf-time-semiform-stab-4}
&\exc{\big(\bG[\bu^n]\commentone{\Delta W_{n}},\bv^{n+1}\big)}=\exc{\big(\bG[\bu^n]\commentone{\Delta W_{n}},\bv^{n+1}-\bv^n\big)}\\
&\hskip-2mm\quad\leq\frac{1}{4}\exc{\normbf{\bv^{n+1}-\bv^n}}+k\exc{\normbf{\bG[\bu^n]}}\notag\\
&\hskip-2mm\quad\leq\frac{1}{4}\exc{\normbf{\bv^{n+1}-\bv^n}}+CC_B^2k\exc{\normdel{\bu^n}}+2C_A^2k\notag.
\end{align}

Similar to the estimation of \eqref{eqn:prf-stability-uh1-5}, the second term $I_2$ can be bounded by
\begin{align}\label{eqn:prf-time-semiform-stab-5}
&k\exc{\big(\bF[\bu^n],\bv^{n+1}\big)}\\
&\qquad\leq k\exc{\normbf{\bv^{n+1}}}+CC_B^2k\exc{\normdel{\bu^n}}+2C_A^2k.\notag
\end{align}
	
The left-hand side of \eqref{eqn:prf-time-semiform-stab-3} can be written as
\begin{align}\label{eqn:prf-time-semiform-stab-6}
&\half\mathbb{E}\Bigl[\big(\normbf{\bv^{n+1}}-\normbf{\bv^n}+\normbf{\bv^{n+1}-\bv^n}\big)\Bigr]-\mathbb{E}\Bigl[(\bu^{n+1}-\bu^n,\opL\bu^{n+\half})\Bigr]\\
&\qquad =\half\mathbb{E}\Bigl[(\normbf{\bv^{n+1}}-\normbf{\bv^n}+\normbf{\bv^{n+1}-\bv^n})\Bigr]\notag\\
&\qquad\quad +\mathbb{E}\Bigl[\half\big(\normdel{\bu^{n+1}}\big)\Bigr]\notag\\
&\qquad\quad -\mathbb{E}\Bigl[\half\big(\normdel{\bu^n}\big)\Bigr]\notag\\
&\qquad=\mathbb{E}\Bigl[\energy{n+1}\Bigr]-\mathbb{E}\Bigl[\energy{n}\Bigr]+\half\mathbb{E}\Bigl[\normbf{\bv^{n+1}-\bv^n}\Bigr].\notag
\end{align}	
		
Putting \eqref{eqn:prf-time-semiform-stab-4}-\eqref{eqn:prf-time-semiform-stab-6} together yields
\begin{align}\label{eqn:prf-time-semiform-stab-7}
&\exc{\energy{n+1}}+\frac{1}{4}\sum_{l=0}^{n}\exc{\normbf{\bv^{l+1}-\bv^l}}\\
&\qquad\leq\exc{\energy{0}}+4C_A^2+CC_B^2k\sum_{l=0}^{n}\exc{\energy{l}}.\notag
\end{align}
		
By the discrete Gronwall's inequality, we obtain
\begin{equation}\label{eqn:prf-time-semiform-stab-8}
\exc{\energy{n+1}}\leq \bigl(\exc{J(\bu^0, \bv^0)}+4C_A^2\bigr)e^{CC_B^2}.
\end{equation}
The proof is complete.
\end{proof}

Since previous stability results are not sufficient to establish the convergence results of fully discrete finite element methods, we will prove some stability results in stronger norms. Consider the following energy functional:
\begin{equation}\label{20220604_1}
\widetilde{J}(\bu^n, \bv^n)=\half\big(\normbf{\opL\bu^n}+\normdel{\bv^n}\big),
\end{equation}	
we will establish the stability for the expectation of the above energy functional.	

\begin{lemma}\label{lemma:time-semiform-stab-uh2-vh1}
Let $(\bu^0, \bv^0)=(\bu_0, \bv_0)\in (\bH^2\cap\bH^1_0)\times\bH^1_0$ be given. The solution $\{(\bu^n, \bv^n);1\leq n\leq N\}$ of \eqref{eqn:time-semiform}--\eqref{20220529_1} satisfies
\begin{equation}\label{eqn:time-semiform-stab-uh2-vh1}
\max_{1\leq n\leq N}\exc{\widetilde{J}(\bu^{n}, \bv^{n})}\leq\left(CC_A^2C_{s1}+\exc{\widetilde{J}(\bu^{0}, \bv^{0})}\right)e^{T}.
\end{equation}
\end{lemma}
	
\begin{proof}\label{prf:time-semiform-stab-uh2-vh1}
By choosing $\bphi=\opL^2\bu^{n+\half}$ in \eqref{eqn:time-semiform}, we get
\begin{equation}\label{eqn:prf-time-semiform-stab-uh2-vh1-1}
(\opL\bu^{n+1}-\opL\bu^n, \opL\bu^{n+\half})=k(\opL\bv^{n+1}, \opL\bu^{n+\half}).
\end{equation}
	
By choosing $\bpsi=-\opL\bv^{n+1}$ in \eqref{20220529_1}, we get
\begin{align}\label{eqn:prf-time-semiform-stab-uh2-vh1-2}
&(\bv^{n+1}-\bv^{n}, -\opL\bv^{n+1})=k(\opL\bu^{n+\half}, -\opL\bv^{n+1})\\
&\qquad +\lambda\big(\Div(\bG[\bu^n]\commentone{\Delta W_{n}}), \Div(\bv^{n+1})\big) + \mu\big(\veps(\bG[\bu^n]\commentone{\Delta W_{n}}), \veps(\bv^{n+1})\big)\notag\\
&\qquad +k\lambda\big(\Div(\bF[\bu^n]), \Div(\bv^{n+1})\big)+\mu\big(\veps(\bF[\bu^n]), \veps(\bv^{n+1})\big).\notag
\end{align}
	
Combining \eqref{eqn:prf-time-semiform-stab-uh2-vh1-1} and \eqref{eqn:prf-time-semiform-stab-uh2-vh1-2} yields
\begin{align}\label{eqn:prf-time-semiform-stab-uh2-vh1-3}
&\widetilde{J}\big(\bu^{n+1}, \bv^{n+1})-\widetilde{J}(\bu^{n}, \bv^{n})\\
&\qquad\quad +\half\left(\normdel{\bv^{n+1}-\bv^n}\right)\notag\\
&\qquad=\innerprd{\bG[\bu^n]\commentone{\Delta W_{n}}}{\bv^{n+1}}\notag\\
&\qquad\quad+k\innerprd{\bF[\bu^n]}{\bv^{n+1}}\notag\\
&\qquad:=I_1 + I_2 + I_3 + I_4.\notag
\end{align}
		
Note that
\begin{equation}\label{20220604_2}
\exc{\innerprd{\bG[\bu^n]\commentone{\Delta W_{n}}}{\bv^{n}}}=0.
\end{equation}

Similar to the estimation of equation \eqref{eqn:prf-stability-uh2-5}, and by It\^o's isometry, we have
\begin{align}\label{eqn:prf-time-semiform-stab-uh2-vh1-4}
I_1 + I_2 &\leq \frac{k}{2}\left(\normdel{\bG[\bu^n]}\right)\\
&\qquad+ \frac{1}{2}\left(\normdel{\bv^{n+1}-\bv^n}\right)\notag\\
&\leq \frac{1}{2}\left(\normdel{\bv^{n+1}-\bv^n}\right)+CC_A^2C_{s1}k.\notag
\end{align}
		
Similar to the estimation of equation \eqref{eqn:prf-stability-uh2-4}, we have
\begin{equation}\label{eqn:prf-time-semiform-stab-uh2-vh1-5}
I_3 + I_4 \leq\frac{k}{2}\left(\normdel{\bv^{n+1}}\right)+CC_A^2C_{s1}k.
\end{equation}

 Taking the expectation and summation over n from 0 to $l$ on the both sides of \eqref{eqn:prf-time-semiform-stab-uh2-vh1-3}, using equations \eqref{20220604_2}--\eqref{eqn:prf-time-semiform-stab-uh2-vh1-5}, and then switching $n$ and $l$, we have
\begin{align}\label{eqn:prf-time-semiform-stab-uh2-vh1-6}
&\exc{\widetilde{J}(\bu^{n+1}, \bv^{n+1})}\leq CC_A^2C_{s1}+\exc{\widetilde{J}(\bu^{0}, \bv^{0})}+k\sum\limits_{l=0}^{n}\exc{\tilde{J}(\bu^{l}, \bv^{l})}.
\end{align}
		
Then the discrete Gronwall's inequality yields
\begin{equation}\label{eqn:prf-time-semiform-stab-uh2-vh1-7}
\exc{\widetilde{J}(\bu^{n+1}, \bv^{n+1})}\leq\left(CC_A^2C_{s1}+\exc{\widetilde{J}(\bu^{0}, \bv^{0})}\right)e^{T}.
\end{equation}
	
The proof is complete.
\end{proof}

\subsection{Error estimates for the time semi-discrete scheme}\label{sec:error-estimate-time}
In this section, we derive error estimates in both ${\bf H}^1$-norm and ${\bf L}^2$-norm for the time semi-discrete scheme. 

\begin{theorem}\label{theo:time-semiform-order-h1}
Let $(\bu, \bv)$ be the solution of \eqref{eqn:weakform-1}--\eqref{eqn:weakform-2} and $\{(\bu^n, \bv^n); 1\leq n\leq N\}$ be the solution of \eqref{eqn:time-semiform}--\eqref{20220529_1}. Under the assumptions \eqref{20220528_1}--\eqref{20220528_4}, there holds
\begin{equation}\label{eqn:time-semiform-order-h1}
\max_{1\leq n\leq N}\exc{\normbfh{\bu(t_{n})-\bu^n}+\normbf{\bv(t_{n})-\bv^n}}\leq C_{s6}e^{CC_B^2}k.
\end{equation}
\end{theorem}
	
\begin{proof}\label{prf:time-semiform-order-h1}
By \eqref{eqn:weakform-1}--\eqref{eqn:weakform-2}, we have
\begin{align}\label{20220529_6}
\big(\bu(t_{n+1})-\bu(t_n), \bphi\big)=&\big(\int_{t_n}^{\commentone{t_{n+1}}}\bv(s)\td s, \bphi\big),\\
\big(\bv(t_{n+1})-\bv(t_n), \bpsi\big)=&\inttn\big(\opL(\bu(s)), \bpsi\big)\td s+\inttn\big(\bF[\bu(s)], \bpsi\big)\td s\label{20220529_7}\\
&+\left(\inttn\bG[\bu(s)]\td W(s), \bpsi\right)\nonumber.
\end{align}
		
By \eqref{eqn:time-semiform}--\eqref{20220529_1}, we have
\begin{alignat}{2}
\big(\bu^{n+1}-\bu^n, \bphi\big) = & k\big(\bv^{n+1}, \bphi\big) 
&&\forall \bphi \in \bL^2,\label{20220529_4}\\  
(\bv^{n+1}-\bv^n, \bpsi) = &\innerprdminusk{\bu^{n+\half}}{\bpsi} &\label{20220529_5}\\
&  + \big(\bG[\bu^n]\Delta_{n+1} W, \bpsi\big)+ k\big(\bF[\bu^n], \bpsi\big)	&&\forall \bpsi \in \bH^1_0.\notag
\end{alignat}

Define notations $\bferr{u}{n}$, $\bferr{v}{n}$, and $\bferr{u}{n+\frac12}$ by
\begin{align*}
\bferr{u}{n}&:=\bu(t_{n+1})-\bu^{n+1}, \quad\bferr{v}{n}:=\bv(t_{n+1})-\bv^{n+1}, \\
\quad\bferr{u}{n+\frac12}&:=\frac{\bu(t_{n+1})+\bu(t_n)}{2}-\bu^{n+\frac12} = \frac{\bferr{u}{n+1}+\bferr{u}{n}}{2}.
\end{align*}

Subtracting \eqref{20220529_4} from \eqref{20220529_6} and \eqref{20220529_5} from \eqref{20220529_7} yield
\begin{align}
&\big(\bferr{u}{n+1}-\bferr{u}{n}, \bphi\big)=\left(\inttn\big(\bv(s)-\bv(t_{n+1})\big)\td s, \bphi\right)+k\big(\bferr{v}{n+1}, \bphi\big),\label{eqn:prf-time-semiform-order-h1-2}\\
&\big(\bferr{v}{n+1}-\bferr{v}{n}, \bpsi\big)=\inttn\left(\opL\Big(\bu(s)-\frac{\bu(t_{n+1})+\bu(t_n)}{2}\Big), \bpsi\right)\td s\label{eqn:prf-time-semiform-order-h1-3}\\
&\quad\qquad\qquad\qquad +k\innerlap{\bferr{u}{n+\half}}{\bpsi}+\inttn\big(\bF[\bu(s)]-\bF[\bu^n], \bpsi\big)\td s\notag\\
&\quad\qquad\qquad\qquad +\left(\inttn\big(\bG[\bu(s)]-\bG[\bu^n]\big)\td W(s), \bpsi\right).\nonumber
\end{align}
		
By choosing $\bphi = -\opL\mbf{e}_\bu^{n+\half}$ in \eqref{eqn:prf-time-semiform-order-h1-2}, we have
\begin{align}
k\innerlap{\mbf{e}_\bu^{n+\half}}{\bferr{v}{n+1}}=&\bigl(\bferr{u}{n+1}-\bferr{u}{n},\opL\bferr{u}{n+\half}\bigr) -\left(\inttn\big(\bv(s)-\bv(t_{n+1})\big)\td s, \opL\bferr{u}{n+\frac12}\right).\label{eqn:prf-time-semiform-order-h1-4}
\end{align}
	
By choosing $\bpsi=\bferr{v}{n+1}$ in \eqref{eqn:prf-time-semiform-order-h1-3}, we obtain 
\begin{align}\label{eqn:prf-time-semiform-order-h1-5}
\big(\bferr{v}{n+1}-\bferr{v}{n}, \bferr{v}{n+1}&\big)=\inttn\left(\opL\Big(\bu(s)-\frac{\bu(t_{n+1})+\bu(t_n)}{2}\Big), \bferr{v}{n+1}\right)\td s\\
&+k\innerlap{\bferr{u}{n+\half}}{\bferr{v}{n+1}}+\inttn\big(\bF[\bu(s)]-\bF[\mbf{u^n}], \bferr{v}{n+1}\big)\td s\notag\\
&+\left(\inttn(\bG[\bu(s)]-\bG[\bu^n])\td W(s), \bferr{v}{n+1}\right).\notag
\end{align}

Combining \eqref{eqn:prf-time-semiform-order-h1-4}-\eqref{eqn:prf-time-semiform-order-h1-5} yields
\begin{align}\label{eqn:prf-time-semiform-order-h1-6}
&\engerr{n+1}-\engerr{n}+\half\normbf{\bferr{v}{n+1}-\bferr{v}{n}}\\
&\qquad=-\inttn\innerlap{\bferr{u}{n+\frac12}}{\bv(s)-\bv(t_{n+1})}\td s\notag\\
&\qquad\quad+\inttn\left(\opL\Big(\bu(s)-\frac{\bu(t_{n+1})+\bu(t_n)}{2}\Big), \bferr{v}{n+1}\right)\td s\notag\\
&\qquad\quad+k\big(\bF[\bu(t_n)]-\bF[\bu^n], \bferr{v}{n+1}\big)+\inttn\big(\bF[\bu(s)]-\bF[\bu(t_n)], \bferr{v}{n+1}\big)\td s\notag\\
&\qquad\quad+\left(\inttn\big(\bG[\bu(s)]-\bG[\bu^n]\big)\td W(s), \bferr{v}{n+1}\right)\notag\\
&\qquad:=I_1+I_2+I_3+I_4+I_5.\notag
\end{align}

By equation \eqref{eqn:holder-cont-v-2}, the expectation of the first term $I_1$ can be bounded by
\begin{align}\label{eqn:prf-time-semiform-order-h1-7}
\exc{I_1}\leq&2\inttn\mathbb{E}\Bigl[\lambda\bigl\|\mathrm{div}\bigl({\bf v}(s)-{\bf v}(t_{n+1})\bigr)\bigr\|_{{\bf L}^2}^{2}\\
&+\mu\bigl(\epsilon\bigl({\bf v}(s)-{\bf v}(t_{n+1})\bigr)\bigr\|_{{\bf L}^2}^{2}\Bigr]\td s\notag\\
&+\frac{k}{4}\exc{\normdel{\bferr{u}{n+1}}}\notag\\
&+\frac{k}{4}\exc{\normdel{\bferr{u}{n}}}\notag\\
\leq&C_{s6}k^2+\frac{k}{4}\exc{\normdel{\bferr{u}{n+1}}}\notag\\
&+\frac{k}{4}\exc{\normdel{\bferr{u}{n}}}.\notag
\end{align}

By equation \eqref{eqn:holder-cont-u-3}, the expectation of the second term $I_2$ can be bounded by
\begin{align}\label{eqn:prf-time-semiform-order-h1-8}
\exc{I_2}\leq&\frac{1}{8} \int_{t_{n}}^{t_{n+1}} \mathbb{E}\left[\left\|\mathcal{L}\left(\mathbf{u}(s)-\mathbf{u}\left(t_{n}\right)\right)\right\|_{\mathbf{L}^{2}}^{2}\right] \mathrm{d} s\\
&+\frac{1}{8} \int_{t_{n}}^{t_{n+1}} \mathbb{E}\left[\left\|\mathcal{L}\left(\mathbf{u}(s)-\mathbf{u}\left(t_{n+1}\right)\right)\right\|_{\mathbf{L}^{2}}^{2}\right] \mathrm{d} s+k \mathbb{E}\left[\left\|\mathbf{e}_{\mathbf{v}}^{n+1}\right\|_{\mathbf{L}^{2}}^{2}\right] \notag\\
\leq & \frac{1}{12} C_{s4} k^{3}+k \mathbb{E}\left[\left\|\mathbf{e}_{\mathbf{v}}^{n+1}\right\|_{\mathbf{L}^{2}}^{2}\right].\notag
\end{align}

By the Poincar\'{e} inequality, the Korn's inequality, and equations \eqref{eqn:holder-cont-u-1} and \eqref{eqn:holder-cont-u-2}, we have
\begin{align}\label{eqn:prf-time-semiform-order-h1-9}
\exc{I_3}\leq&\frac{k}{4}\exc{\normbf{\bF[\bu(t_n)]-\bF[\bu^n]}}+k\exc{\normbf{\bferr{v}{n+1}}}\\
\leq&\frac{k}{4}C_B^2\exc{\lambda\|\Div(\bferr{u}{n})\|^{2}_{\bL_{2}}+\mu\|\veps(\bferr{u}{n})\|^{2}_{\bL_{2}}+\|\bferr{u}{n}\|^{2}_{\bL_{2}}}+k\exc{\normbf{\bferr{v}{n+1}}}\notag\\
\leq&kCC_B^2\exc{\lambda\|\Div(\bferr{u}{n})\|^{2}_{\bL_{2}}+\mu\|\veps(\bferr{u}{n})\|^{2}_{\bL_{2}}}+k\exc{\normbf{\bferr{v}{n+1}}},\notag\\
\exc{I_4}\leq&\inttn\exc{\normbf{\bF[\bu(s)]-\bF[\bu(t_n)]}} \td s+k\exc{\normbf{\bferr{v}{n+1}}}\label{eqn:prf-time-semiform-order-h1-10}\\
\leq& \frac14C_B^2\int_{t_n}^{t_{n+1}}\Big(\exc{\lambda\normbf{\Div(\bu(s))-\Div(\bu(t_n))}}\notag\\
&+\exc{\mu\normbf{\veps(\bu(s))-\veps(\bu(t_n))}}+\exc{\normbf{\bu(s)-\bu(t_n)}}\Big)ds+k\exc{\normbf{\bferr{v}{n+1}}}\notag\\
\leq&CC_B^2k^3(C_{s4}+C_{s2})+k\exc{\normbf{\bferr{v}{n+1}}}.\notag
\end{align}

For the term $I_5$, note that $\exc{\left(\inttn\big(\bG[\bu(s)]-\bG[\bu^n]\big)\td W(s), \bferr{v}{n}\right)}=0$. Then by the Poincar\'{e} inequality, the Korn's inequality, It\^o's isometry, and equations \eqref{eqn:holder-cont-u-1} and \eqref{eqn:holder-cont-u-2}, we have		
\begin{align}\label{eqn:prf-time-semiform-order-h1-11}
\exc{I_5}=&\exc{\big((\bG[\bu(t_n)]-\bG[\bu^n])\commentone{\Delta W_{n}}, \bferr{v}{n+1}-\bferr{v}{n}\big)}\\
& +\exc{\left(\inttn\big(\bG[\bu(s)]-\bG[\bu(t_n)]\big)\td W(s), \bferr{v}{n+1}-\bferr{v}{n}\right)}\notag\\
\leq&\frac{1}{4}\exc{\normbf{\bferr{v}{n+1}-\bferr{v}{n}}}+2kC_B^2\exc{\lambda\normbf{\Div(\bferr{u}{n})}+\mu\normbf{\veps(\bferr{u}{n})}+\normbf{\bferr{u}{n}}}\notag\\
& +2C_B^2\inttn\mathbb{E}\Bigl[\lambda\normbf{\Div(\bu(s)-\bu(t_n))}+\mu\normbf{\veps(\bu(s)-\bu(t_n))}\Bigr]\td s\notag\\
& +2C_B^2\inttn\exc{\normbf{\bu(s)-\bu(t_n)}}\td s\notag\\
\leq&\frac{1}{4}\exc{\normbf{\bferr{v}{n+1}-\bferr{v}{n}}}+CC_B^2k\exc{\engerr{n}}+CC_B^2k^3(C_{s1}+C_{s2}).\notag
\end{align}
		
Combining \eqref{eqn:prf-time-semiform-order-h1-6}-\eqref{eqn:prf-time-semiform-order-h1-11} together, we have
\begin{align}
&\exc{\engerr{n+1}}-\exc{\engerr{n}}+\frac{1}{4}\exc{\normbf{\bferr{v}{n+1}-\bferr{v}{n}}}\leq CC_B^2k\exc{\engerr{n}}+C_{s6}k^2.\label{eqn:prf-time-semiform-order-h1-12}
\end{align}
		
Then the discrete Gronwall's inequality yields
\begin{equation}\label{eqn:prf-time-semiform-order-h1-13}
\exc{\engerr{n+1}}\leq C_{s6}e^{CC_B^2}k.
\end{equation}
where we use the fact that $\exc{\energy{0}}=0$. Then the theorem is proved.
\end{proof}

Theorem \ref{theo:time-semiform-order-h1} states that the time semi-discrete convergence order in $\bH^1$-norm is $\half$. The next theorem establishes the time semi-discrete $\bL^2$-error of $\bu^n$, and it states that the time semi-discrete convergence order in $\bL^2$-norm is $1$.
\commentone{Its proof is inspired by a similar proof for the stochastic scalar wave equation given  in \cite{feng2022higher}.}
	
\commentone{
\begin{theorem}\label{theo:time-semiform-order-l2}
Let $(\bu, \bv)$ be the solution of \eqref{eqn:weakform-1}--\eqref{eqn:weakform-2} and $\{\bu^n, \bv^n\}_n$ be the solution of \eqref{eqn:time-semiform}--\eqref{20220529_1}. Under the assumptions \eqref{20220528_1}--\eqref{20220528_4} and the assumptions that
\begin{align*}
\|\mbf{G}[\bu(0)]\|_{L^2}^2\le Ck\qquad\text{and}\qquad \|k\bv_0-(\bu^1-\bu^0)\|_{L^2}^2\le Ck^4,
\end{align*}
there holds		
\begin{equation}\label{eqn:time-semiform-order-l2}
\max_{1\leq n\leq N}\exc{\normbf{\bu(t_{n})-\bu^n}}\leq CC_{s4}k^2.
\end{equation}
\end{theorem}
}

\begin{proof}\label{prf:time-semiform-order-l2}
Note that $d_t\bu^l = \frac{(\bu^l-\bu^{l-1})}{k}$. Then by \eqref{20220529_1}, we have
		
\begin{align}\label{eqn:prf-time-semiform-order-l2-1}
\big(d_t\bu^{l+1}-d_t\bu^l,& \bpsi\big)+k\lambda\big(\Div\bu^{l+\half}, \Div\bpsi\big) + k\mu\big(\varepsilon(\bu^{l+\half}), \varepsilon(\bpsi)\big)\\
&=\big(\mbf{G}[\bu^l]\Delta W_{l}, \bpsi\big) + k\big(\mbf{F}[\bu^l], \bpsi\big).\notag
\end{align}
		
Denote $\bar{\bu}^{n+\frac12}=\sum\limits_{l=1}^{n}\bu^{l+\frac12}$. Taking the summation over $l$ from $l=1$ to $l=n$ and multiplying $k$ on both sides of \eqref{eqn:prf-time-semiform-order-l2-1} yield
\begin{align}\label{eqn:prf-time-semiform-order-l2-2}
\big(\bu^{n+1}-\bu^n,& \bpsi\big)-k^2\innerlap{\bar{\bu}^{n+\half}}{\bpsi}-k(d_t\bu^1, \bpsi)\\
&=k\left(\sum_{l=1}^{n}\mbf{G}[\bu^l]\Delta W_{l}, \bpsi\right) + k^2\left(\sum_{l=1}^{n}\mbf{F}[\bu^l], \bpsi\right).\notag
\end{align}

\commentone{
Besides, using \eqref{eqn:weakform-1} and \eqref{eqn:weakform-2}, we get
\begin{align}\label{eqn:prf-time-semiform-order-l2-3}
&\big(\bu(t_{n+1})-\bu(t_n), \bpsi\big) - \int_{t_n}^{t_{n+1}}\int_{0}^{s}\innerlap{\bu(\xi)}{\bpsi}\td\xi\td s-k\big(\bv_0, \bpsi\big)\\
&\qquad =\left(\int_{t_n}^{t_{n+1}}\int_{0}^{s}\mbf{G}[\bu(\xi)]\td W(\xi)\td s, \bpsi\right) + \int_{t_n}^{t_{n+1}}\int_{0}^{s}\big(\mbf{F}[\bu(\xi)], \bpsi\big)\td\xi\td s.\notag
\end{align}
}

\commentone{
Subtracting \eqref{eqn:prf-time-semiform-order-l2-2} from \eqref{eqn:prf-time-semiform-order-l2-3} lead to
\begin{align}\label{eqn:prf-time-semiform-order-l2-4}
&\big(\bferr{u}{n+1}-\bferr{u}{n}, \bpsi\big)-k^2\innerlap{\bar{\bf e}_{\bf u}^{n+\half}}{\bpsi}\\
&\qquad=\bigg(\int_{t_n}^{t_{n+1}}\int_{0}^{s}\opL\bu(\xi)\td\xi\td s-k^2\sum\limits_{l=1}^{n}\opL\bigl(\frac{\bu(t_{l+1})+\bu(t_l)}{2}\bigr),\bpsi\bigg)\notag\\ 
&\qquad\quad+ \left(\int_{t_n}^{t_{n+1}}\int_{0}^{s}\mbf{G}[\bu(\xi)]\td W(\xi)\td s - k\sum_{l=1}^{n}\mbf{G}[\bu^l]\Delta W_{l}, \bpsi\right)\notag\\
&\qquad\quad+\left(\int_{t_n}^{t_{n+1}}\int_{0}^{s}\mbf{F}[\bu(\xi)]\td\xi\td s - k^2\sum_{l=1}^{n}\mbf{F}[\bu^l], \bpsi\right)\notag\\
&\qquad\quad-\bigl(k\bv_0-(\bu^1-\bu^0), \bpsi\bigr).\notag
\end{align}
}

\commentone{
Borrowing an idea of \cite{feng2022higher}, we write  $\int_{0}^{s}\td\xi=\sum\limits_{l=0}^{n-1}\int_{t_{l}}^{t_{l+1}}\td\xi+\int_{t_n}^{s}\td\xi$ and choosing $\bpsi = \bferr{u}{n+\half}$ in \eqref{eqn:prf-time-semiform-order-l2-4} yield
\begin{align}\label{eqn:prf-time-semiform-order-l2-5}
&\half\Big(\normbf{\bferr{u}{n+1}}-\normbf{\bferr{u}{n}}\Big)\\
& \qquad+\frac{k^2\lambda}{2}\Big(\normbf{\Div(\bar{\bf e}_{\bf u}^{n+\half})}-\normbf{\Div(\bar{\bf e}_{\bf u}^{n-\half})}+\normbf{\Div(\bferr{u}{n+\half})}\Big)\notag\\
& \qquad+ \frac{k^2\mu}{2}\Big(\normbf{\veps(\bar{\bf e}_{\bf u}^{n+\half})}-\normbf{\veps(\bar{\bf e}_{\bf u}^{n-\half})}+\normbf{\veps(\bferr{u}{n+\half})}\Big)\notag\\
&:=I_1 + I_2 + \cdots +I_{15}, \notag
\end{align}
}

\commentone{
where
\begin{align}
&I_1+\cdots+I_4=\frac12\int_{t_n}^{t_{n+1}}\sum_{l = 1}^{n}\int_{t_{l-1}}^{t_{l+1}}\left(\opL\bu(\xi)-\opL\bigl(\frac{\bu(t_{l+1})+\bu(t_l)}{2}\bigr), \bferr{u}{n+\half}\right)\td\xi\td s\\ 
&\quad+\frac12\int_{t_n}^{t_{n+1}}\int_{t_{0}}^{t_{1}}\innerlap{\bu(\xi)}{\bferr{u}{n+\half}}\td\xi\td s+\frac12\int_{t_n}^{t_{n+1}}\int_{t_n}^{s}\innerlap{\bu(\xi)}{\bferr{u}{n+\half}}\td\xi\td s\notag\\
&\quad-\frac12\int_{t_n}^{t_{n+1}}\int_{s}^{t_{n+1}}\innerlap{\bu(\xi)}{\bferr{u}{n+\half}}\td\xi\td s\notag,
\end{align}
}

\commentone{
\begin{align}
&I_5+\cdots+I_9:=\frac12\int_{t_n}^{t_{n+1}}\sum_{l = 1}^{n}\int_{t_{l-1}}^{t_{l+1}}\big(\mbf{F}[\bu(\xi)]-\mbf{F}[\bu(t_l)], \bferr{u}{n+\half}\big)\td\xi\td s\\
&\quad+k^2\left(\sum_{l= 1}^{n}\big(\mbf{F}[\bu(t_l)]-\mbf{F}[\bu^l]\big), \bferr{u}{n+\half}\right)+\frac12\int_{t_n}^{t_{n+1}}\int_{t_{0}}^{t_{1}}\big(\mbf{F}[\bu(\xi)], \bferr{u}{n+\half}\big)\td\xi\td s\notag\\
&\quad+\frac12\int_{t_n}^{t_{n+1}}\int_{t_{n}}^{s}\big(\mbf{F}[\bu(\xi)], \bferr{u}{n+\half}\big)\td\xi\td s-\frac12\int_{t_n}^{t_{n+1}}\int_{s}^{t_{n+1}}\big(\mbf{F}[\bu(\xi)], \bferr{u}{n+\half}\big)\td\xi\td s\notag,
\end{align}
}

\commentone{
and
\begin{align}
&I_{10}+\cdots+I_{14}:=\left(k\sum_{l = 1}^{n}\big(\mbf{G}[\bu(t_l)]-\mbf{G}[\bu^l]\big) \Delta W_l, \bferr{u}{n+\half}\right)\\
&\quad+\left(\int_{t_n}^{t_{n+1}} \sum_{\ell=1}^{n} \int_{t_{\ell}}^{t_{\ell+1}}\mbf{G}[\bu(\xi)]-\mbf{G}[\bu(t_l)] \mathrm{d} W(\xi) \mathrm{d} s, \bferr{u}{n+\half}\right)\notag \\
&\quad-\left(\int_{t_n}^{t_{n+1}} \int_{s}^{t_{n+1}}\mbf{G}[\bu(\xi)]-\mbf{G}[\bu(t_n)] \mathrm{d} W(\xi) \mathrm{d} s, \bferr{u}{n+\half}\right)\notag\\
&\quad-\left(\int_{t_n}^{t_{n+1}} \int_{s}^{t_{n+1}}\mbf{G}[\bu(t_n)] \mathrm{d} W(\xi) \mathrm{d} s, \bferr{u}{n+\half}\right)+\left(\int_{t_n}^{t_{n+1}} \int_{t_{0}}^{t_{1}}\mbf{G}[\bu(\xi)] \mathrm{d} W(\xi) \mathrm{d} s, \bferr{u}{n+\half}\right)\notag.
\end{align}
}

\commentone{
By equation \eqref{eqn:holder-cont-u-3}, the first term $I_1$ can be bounded by
\begin{align}\label{eqn:prf-time-semiform-order-l2-6}
\exc{I_1}\leq&\int_{t_n}^{t_{n+1}}\sum_{l = 1}^{n}\int_{t_{l-1}}^{t_{l+1}}\exc{\bigg(\opL\bigl(\frac{\bu(\xi)-\bu(t_{l+1})+\bu(\xi)-\bu(t_l)}{2}\bigr), \bferr{u}{n+\half}\bigg)}\td\xi\td s\\
 \leq&\int_{t_n}^{t_{n+1}}\sum_{l = 1}^{n}\int_{t_{l-1}}^{t_{l+1}}\left(\exc{\normbf{\opL\bigl(\frac{\bu(\xi)-\bu(t_{l+1})+\bu(\xi)-\bu(t_l)}{2}\bigr)}}\right)^\half\notag\\
 &\left(\exc{\normbf{\bferr{u}{n+\half}}}\right)^\half\td\xi\td s\notag\\
 \leq&\int_{t_n}^{t_{n+1}}\sum_{l = 1}^{n}\int_{t_{l-1}}^{t_{l+1}}\left(CC_{s4}k^2\exc{\normbf{\bferr{u}{n+\half}}}\right)^\half\td\xi\td s\notag\\
 \leq&k\exc{\normbf{\bferr{u}{n+1}}}+k\exc{\normbf{\bferr{u}{n}}}+CC_{s4}k^3.\notag
\end{align}
}

\commentone{
By equation \eqref{eqn:holder-cont-u-3}, the second term $I_2$ can be bounded by
\begin{align}
\exc{I_2}\le&\frac12\int_{t_n}^{t_{n+1}}\exc{\int_{t_{0}}^{t_{1}}\innerlap{\bigl(\bu(\xi)-\bu(t_0)\bigr)}{\bferr{u}{n+\frac12}}\td\xi}\td s\\
&\quad+\frac12\int_{t_n}^{t_{n+1}}\exc{\int_{t_{0}}^{t_{1}}\innerlap{\bigl(\bu(t_0)\bigr)}{\bferr{u}{n+\frac12}}\td\xi}\td s\notag\\
\le&k\exc{\normbf{\bferr{u}{n+1}}}+k\exc{\normbf{\bferr{u}{n}}}+CC_{s4}k^5+Ck^3\notag.
\end{align}
}

\commentone{
By equation \eqref{eqn:holder-cont-u-3}, the summation of the third term and the fourth term can be bounded by
\begin{align}
\exc{I_3+I_4}\le&\frac12\int_{t_n}^{t_{n+1}}\int_{t_n}^{s}\innerlap{\bigl(\bu(\xi)-\bu(t_n)\bigr)}{\bferr{u}{n+\half}}\td\xi\td s\\
&\quad-\frac12\int_{t_n}^{t_{n+1}}\int_{s}^{t_{n+1}}\innerlap{\bigl(\bu(\xi)-\bu(t_n)\bigr)}{\bferr{u}{n+\half}}\td\xi\td s\notag\\
&\quad+\frac12\innerlap{\bigl(\bu(t_n)\bigr)}{\bferr{u}{n+\half}}\bigl(\int_{t_n}^{t_{n+1}}\int_{t_n}^{s}\td\xi\td s-\int_{t_n}^{t_{n+1}}\int_{s}^{t_{n+1}}\td\xi\td s\bigr)\notag\\
\le&k\exc{\normbf{\bferr{u}{n+1}}}+k\exc{\normbf{\bferr{u}{n}}}+CC_{s4}k^5\notag.
\end{align}
}

\commentone{
By equations \eqref{20220528_2} and \eqref{eqn:holder-cont-u-1}, we have
\begin{align}\label{eqn:prf-time-semiform-order-l2-8}
\exc{I_5}&\leq C\int_{t_n}^{t_{n+1}}\sum_{l = 1}^{n}\int_{t_{l-1}}^{t_{l+1}}\exc{\normbf{\mbf{F}[\bu(\xi)]-\mbf{F}[\bu(t_l)]}}\td\xi\td s\\
&\qquad + k\exc{\normbf{\bferr{u}{n+1}}}+k\exc{\normbf{\bferr{u}{n}}}\notag\\
&\leq C\int_{t_n}^{t_{n+1}}\sum_{l = 1}^{n}\int_{t_{l-1}}^{t_{l+1}}\Bigl[\normbf{\bu(\xi)-\bu(t_l)}\td\xi\td s\notag \\
&\qquad+ k\exc{\normbf{\bferr{u}{n+1}}}+k\exc{\normbf{\bferr{u}{n}}}\notag\\
&\leq CC_{s1}k^3 + k\exc{\normbf{\bferr{u}{n+1}}}+k\exc{\normbf{\bferr{u}{n}}}.\notag
\end{align}
}

\commentone{
By equation \eqref{20220528_2}, the sixth term can be bounded by
\begin{align}\label{eqn:prf-time-semiform-order-l2-10}
\exc{I_6}&=k^2\exc{\left(\sum_{l = 1}^{n}\big(\mbf{F}[\bu(t_l)]-\mbf{F}[\bu^l]\big), \bferr{u}{n+\half}\right)}\\
&\leq k^2\sum_{l = 1}^{n}\left(\frac{T}{4}\exc{\normbf{\mbf{F}[\bu(t_l)]-\mbf{F}[\bu^l]}}+\frac1T\exc{\normbf{\bferr{u}{n+\half}}}\right)\notag\\
&\leq Ck^2\sum_{l = 1}^{n}\mathbb{E}\Bigl[\normbf{\bu(t_l)-\bu^l}\Bigr]+ k\exc{\normbf{\bferr{u}{n+\half}}}\notag\\
&\leq Ck^2\sum_{l = 1}^{n}\exc{\normbf{\bferr{u}{l}}}+k\exc{\normbf{\bferr{u}{n+1}}}+k\exc{\normbf{\bferr{u}{n}}}.\notag
\end{align}
}

\commentone{
By equations \eqref{20220528_2} and \eqref{eqn:holder-cont-u-1}, the seventh term $I_7$ can be bounded by
\begin{align}
\exc{I_7}\le&\frac12\int_{t_n}^{t_{n+1}}\exc{\int_{t_{0}}^{t_{1}}\bigl(\mbf{F}[\bu(\xi)]-\mbf{F}[\bu(t_0)],\bferr{u}{n+\frac12}\bigr)\td\xi}\td s\\
&\quad+\frac12\int_{t_n}^{t_{n+1}}\exc{\int_{t_{0}}^{t_{1}}\bigl(\mbf{F}[\bu(t_0)],\bferr{u}{n+\frac12}\bigr)\td\xi}\td s\notag\\
\le&k\exc{\normbf{\bferr{u}{n+1}}}+k\exc{\normbf{\bferr{u}{n}}}+CC_{s1}k^5+Ck^3\notag.
\end{align}
}

\commentone{
By equation \eqref{20220528_2}, the summation of the third term and the fourth term can be bounded by
\begin{align}
\exc{I_8+I_9}\le&\frac12\int_{t_n}^{t_{n+1}}\int_{t_n}^{s}\bigl(\mbf{F}[\bu(\xi)]-\mbf{F}[\bu(t_n)],\bferr{u}{n+\half}\bigr)\td\xi\td s\\
&\quad-\frac12\int_{t_n}^{t_{n+1}}\int_{s}^{t_{n+1}}\bigl(\mbf{F}[\bu(\xi)]-\mbf{F}[\bu(t_n)],\bferr{u}{n+\half}\bigr)\td\xi\td s\notag\\
&\quad+\frac12\bigl(\mbf{F}[\bu(t_n)],\bferr{u}{n+\half}\bigr)\bigl(\int_{t_n}^{t_{n+1}}\int_{t_n}^{s}\td\xi\td s-\int_{t_n}^{t_{n+1}}\int_{s}^{t_{n+1}}\td\xi\td s\bigr)\notag\\
\le&k\exc{\normbf{\bferr{u}{n+1}}}+k\exc{\normbf{\bferr{u}{n}}}+CC_{s4}k^5\notag.
\end{align}
}

\commentone{
By Young's inequality and equation \eqref{20220528_2}, the tenth term $I_{10}$ can be bounded by
\linespread{1.55}
\begin{align}\label{eqn:prf-time-semiform-order-l2-9}
\exc{I_{10}}&=\exc{k\big(\sum_{l = 1}^{n}\big(\mbf{G}[\bu(t_l)]-\mbf{G}[\bu^l]\big)\Delta W_{l}, \bferr{u}{n+\half}\big)}\\
&\leq k\exc{\normbf{\bferr{u}{ n+\half}}}+Ck\exc{\normbf{\sum_{l = 1}^{n}\big(\mbf{G}[\bu(t_l)]-\mbf{G}[\bu^l]\big)\Delta W_{l}}}\notag\\
&\leq k\exc{\normbf{\bferr{u}{n+1}}}+k\exc{\normbf{\bferr{u}{n}}}+Ck^2\sum_{l = 1}^{n}\mathbb{E}\Bigl[\normbf{\bferr{u}{l}}\Bigr]\notag.
\end{align}
}

\commentone{
For the eleventh term $I_{11}$, by It\^o's isometry, Young's inequality, and equations \eqref{20220528_2} and \eqref{eqn:holder-cont-u-1}, we get
\begin{align}
\exc{I_{11}}&\leq C\int_{t_n}^{t_{n+1}}\sum_{l = 1}^{n}\int_{t_l}^{t_{l+1}}\exc{\normbf{\mbf{G}[\bu(\xi)]-\mbf{G}[\bu(t_l)]}}\td\xi\td s \\
&\qquad+ k\exc{\normbf{\bferr{u}{n+1}}}+k\exc{\normbf{\bferr{u}{n}}}\notag\\
&\leq C\int_{t_n}^{t_{n+1}}\sum_{l = 1}^{n}\int_{t_l}^{t_{l+1}}\mathbb{E}\Bigl[\normbf{\bu(\xi)-\bu(t_l)}\Bigr]\td \xi\td s \notag\\
&\qquad+ k\exc{\normbf{\bferr{u}{n+1}}}+k\exc{\normbf{\bferr{u}{n}}}\notag\\
&\leq CC_{s1}k^3 + k\exc{\normbf{\bferr{u}{n+1}}}+k\exc{\normbf{\bferr{u}{n}}}.\notag
\end{align}
}

\commentone{
By It\^o's isometry, Young's inequality, and equations \eqref{20220528_2} and \eqref{eqn:holder-cont-u-1}, the term $I_{12}$ can be bounded by
\begin{align}
\exc{I_{12}}&\leq C\int_{t_n}^{t_{n+1}}\int_{s}^{t_{n+1}}\exc{\normbf{\mbf{G}[\bu(\xi)]-\mbf{G}[\bu(t_n)]}}\td\xi\td s \\
&\qquad+ k\exc{\normbf{\bferr{u}{n+1}}}+k\exc{\normbf{\bferr{u}{n}}}\notag\\
&\leq C\int_{t_n}^{t_{n+1}}\int_{s}^{t_{n+1}}\mathbb{E}\Bigl[\normbf{\bu(\xi)-\bu(t_n)}\Bigr]\td \xi\td s \notag\\
&\qquad+ k\exc{\normbf{\bferr{u}{n+1}}}+k\exc{\normbf{\bferr{u}{n}}}\notag\\
&\leq CC_{s1}k^4 + k\exc{\normbf{\bferr{u}{n+1}}}+k\exc{\normbf{\bferr{u}{n}}}.\notag
\end{align}
}

\commentone{
By equations \eqref{eqn:prf-time-semiform-order-h1-2} and \eqref{eqn:holder-cont-v-1}, and Theorem \ref{theo:time-semiform-order-h1}, the term $I_{13}$ can be bounded by
\begin{align}
\exc{I_{13}}&= \mathbb{E}\bigg[\mbf{G}[\bu(t_n)]\bigl(\int_{t_n}^{t_{n+1}}W(t_{n+1})-W(s)\td s,\bferr{u}{n+\half}\bigr)\bigg]\\
&= \frac{1}{2}\mathbb{E}\bigl[\bigl(\mbf{G}[\bu(t_n)](\int_{t_n}^{t_{n+1}}W(t_{n+1})-W(s)\td s,\bferr{u}{n+1}-\bferr{u}{n}\bigr)\bigr]\notag\\
&= \frac{1}{2}\mathbb{E}\bigl[\bigl(\mbf{G}[\bu(t_n)]\int_{t_n}^{t_{n+1}}W(t_{n+1})-W(s)\td s,k\bferr{v}{n+1}\bigr)\bigr]\notag\\
&\quad+\frac{1}{2}\mathbb{E}\bigl[\bigl(\mbf{G}[\bu(t_n)]\int_{t_n}^{t_{n+1}}W(t_{n+1})-W(s)\td s,\int_{t_n}^{t_{n+1}}v(s)-v(t_{n+1})\td s\bigr)\bigr]\notag\\
&\le Ck^3+k^2\|\bferr{v}{n+1}\|_{L^2}^2+k\int_{t_n}^{t_{n+1}}\|v(s)-v(t_{n+1})\|_{L^2}^2\td s\notag\\
&\le Ck^3+C_{s6}e^{CC_B^2}k^3+C_{s5}k^3\notag.
\end{align}
}

\commentone{
By It\^o's isometry, Young's inequality, and equations \eqref{20220528_2} and \eqref{eqn:holder-cont-u-1}, the term $I_{14}$ can be bounded by
\begin{align}\label{eq20220925_1}
\exc{I_{14}}&=\left(\int_{t_n}^{t_{n+1}} \int_{t_{0}}^{t_{1}}\mbf{G}[\bu(\xi)]-\mbf{G}[\bu(0)] \mathrm{d} W(\xi) \mathrm{d} s, \bferr{u}{n+\half}\right)\\
&\quad+\left(\int_{t_n}^{t_{n+1}} \int_{t_{0}}^{t_{1}}\mbf{G}[\bu(0)] \mathrm{d} W(\xi) \mathrm{d} s, \bferr{u}{n+\half}\right)\notag\\
&\le CC_{s1}k^4+k\|\bferr{u}{n+\half}\|_{L^2}^2+Ck^2\|\mbf{G}[\bu(0)]\|_{L^2}^2\notag\\
&\le k\|\bferr{u}{n+\half}\|_{L^2}^2+Ck^3.\notag
\end{align}
}

\commentone{
Combining \eqref{eqn:prf-time-semiform-order-l2-5}-\eqref{eq20220925_1} and taking the summation, we have
\begin{align}\label{eqn:prf-time-semiform-order-l2-11}
\exc{\normbf{\bferr{u}{n+1}}}\le& Ck\sum\limits_{l=0}^n\exc{\normbf{\bferr{u}{l+1}}}+CC_{s4}k^2+CC_{s1}k^2+Ck^2\\
&\quad+C_{s6}e^{CC_B^2}k^2+C_{s5}k^2.\notag
\end{align}
}

Then the discrete Gronwall's inequality yields
\begin{align}\label{eqn:prf-time-semiform-order-l2-12}
\exc{\normbf{\bferr{u}{n+1}}}\leq CC_{s4}k^2.
\end{align}
where we use $\bferr{u}{0}=\bf{0}$.
\end{proof}

\section{Finite element discretization in space}\label{sec:fem-discrete}
In this section we discretize the time semi-discrete scheme in space using the finite element methods and give detailed error analysis for the fully discrete scheme.
\subsection{Finite element fully discrete scheme}\label{ssec:fem-form}
Let $\cal{T}_h$ be a quasi-uniform triangulation of $\cal{D}$ with diameter $h$. We consider the finite element spaces
\begin{align*}
\bUh^{r_1} = \left\{\bu_h\in\bH^1: \bu_h|_{K}\in\mbf{P}_{r_1}({K})\quad \forall K\in\cal{T}_h\right\},\\
\bVh^{r_2} = \left\{\bv_h\in\bH^1: \bv_h|_{K}\in\mbf{P}_{r_2}({K})\quad \forall K\in\cal{T}_h\right\},
\end{align*}
where $\mbf{P}_r({K})$ donates the space of polynomials with degree not exceeding a given integer $r$ on $K\in\cal{T}_h$. Next, we define two types of projection as follows.
\begin{definition}\label{def:projection}
{\rm (1)} The $\bL^2$-projection $\Ph:\bL^2\to\bUh^{r_1}$ is defined by
\begin{align*}
(\Ph\bw, \bv_h) = (\bw, \bv_h),
\end{align*}
for all $\bw\in\bL^2$ and $\bv_h\in\bUh^{r_1}$.
			
{\rm (2)} The $\bH^1$-projection $\Rh:\bH^1_0\to\bVh^{r_2}$ is defined by 
\begin{align*}
\innerprd{\Rh\bw}{\bv_h} = \innerprd{\bw}{\bv_h},
\end{align*}
for all $\bw\in\bH^1_0$ and $\bv_h\in\bVh^{r_2}$.
\end{definition}

The following error estimate results are well-known \cite{baker1976error,cohen2017finite}.
\begin{align}\label{eqn:prj-error-estimate-ph}
\left\lVert{\bw-\Ph\bw}\right\rVert_{\bL^2}\leq& C_{\Ph}h^{\min{(r_1+1, s)}}\left\lVert\bw\right\rVert_{\bH^s},\\
\left\lVert{\bw-\Rh\bw}\right\rVert_{\bL^2}\leq& C_{\Rh}h^{\min{(r_2+1, s)}}\left\lVert\bw\right\rVert_{\bH^s},\label{20220608_2}\\
\Bigl(\lambda\left\lVert{\Div(\bw-\Rh\bw)}\right\rVert^2_{\bL^2}+\mu\left\lVert{\veps(\bw-\Rh\bw)}\right\rVert^2_{\bL^2}\Bigr)^{\half}\leq& C_{\Rh}h^{\min{(r_2, s-1)}}\left\lVert\bw\right\rVert_{\bH^s}.\label{20220608_3}
\end{align}

\begin{scheme}\label{sch:scheme1-h}
Let $(\bu^0_h, \bv^0_h)=(\Rh\bu^0, \Ph\bv^0)$. Seeking a $\bUh^{r_1}\times\bVh^{r_2}$-valued $\{\cal{F}_{t_{n+1}}\}$-adapted solution $(\bu^{n+1}_h, \bv^{n+1}_h)$ such that $\bb{P}$-almost surely
\begin{align}\label{eqn:fem-form}
\big(\bu^{n+1}_h-\bu^n_h, \bphi_h\big) = & k\big(\bv^{n+1}_h, \bphi_h\big),\\
(\bv^{n+1}_h-\bv^n_h, \bpsi_h) = &k\bigl(\opLh\bu^{n+\half}_h,\bpsi_h\bigr) + 
\big(\bG[\bu^n_h]\Delta_{n+1} W, \bpsi_h\big)+ k\big(\bF[\bu^n_h], \bpsi_h\big),\label{20220608_4}
\end{align}
for all $(\bphi_h, \bpsi_h)\in \bUh^{r_1}\times\bVh^{r_2}$.
\end{scheme}

\begin{remark} 
At each time step, the above scheme is a nonlinear random algebraic system for $(\bu^{n+1}_h, \bv^{n+1}_h)$ whose well-posedness can be proved by a standard fixed point argument based on the stability estimates of the next lemma. 
\end{remark}
	
Since the proof of the next lemma is similar to that of Lemma \ref{lemma:time-semiform-stab}, we will omit it to save space.

\begin{lemma}\label{lemma:fem-form-stab}
Assume $0<k<<k_0<T$ and choose $(\bu^0_h, \bv^0_h)=(\Rh\bu^0, \Ph\bv^0)$. Then there holds
\begin{equation}\label{eqn:fem-form-stab}
\max_{1\leq n\leq N-1}\exc{\energyh{n+1}}\leq \bigl(\exc{J(\bu^0, \bv^0)}+4C_A^2\bigr)e^{CC_B^2}.
\end{equation}
\end{lemma}

\subsection{Error estimates for the finite element fully discrete scheme}\label{ssec:error-estimate-fem}
	
The linear finite elements are used in this section, and $r_1 = r_2 = 1$ are chosen in the finite element space. We start with the discretization of the operator $\opL$.
	
\begin{definition}
The discrete operator $\opLh:\bUh^1\to\bVh^1$ is defined by
\begin{align}
\left(-\opLh\bw_h, \bv_h\right)=\innerprd{\bw_h}{\bv_h}.
\end{align}
for all $(\bw_h, \bv_h)\in\bUh^1\times\bVh^1$.
\end{definition}

Define $\bfErr{u}{n} := \bu^n - \bu^n_h$ and $\bfErr{v}{n} := \bv^n - \bv^n_h$, we now derive the estimates for $\bfErr{u}{n}$ and $\bfErr{v}{n}$.

\begin{theorem}\label{theo:fem-form-order-h1}
Suppose $\{\bu^n, \bv^n\}_n$ solves \eqref{eqn:time-semiform}--\eqref{20220529_1} and $\{\bu^n_h, \bv^n_h\}$ solves \eqref{eqn:fem-form}--\eqref{20220608_4}. $(\bu^0_h, \bv^0_h)=(\Rh\bu^0, \Ph\bv^0)$ and $(\bu^0, \bv^0)\in \bH^2\times\bH^1$ are given. Under the assumptions \eqref{20220528_1}--\eqref{20220528_4}, we have
\begin{equation}\label{eqn:fem-form-order-h1}
\max\limits_{1\leq n \leq N}\exc{\normbfh{\bu^n-\bu^n_h}+\normbf{\bv^n-\bv^n_h}}\leq C\exc{\left\lVert\bu^n\right\rVert^2_{\bH^2}}h^2.
\end{equation}
\end{theorem}
	
\begin{proof} 
Subtracting \eqref{eqn:fem-form} from \eqref{eqn:time-semiform} leads to
\begin{align}\label{eqn:prf-fem-form-order-h1-1}
\big(\bfErr{u}{n+1}-\bfErr{u}{n}, \bphi_h\big) = &k\big(\bfErr{v}{n+1}, \bphi_h\big),\\
\big(\bfErr{v}{n+1}-\bfErr{v}{n}, \bpsi_h\big) = &k\big(\opL\bfErr{u}{n+\half}, \bpsi_h\big)+\big((\bG[\bu^n]-\bG[\bu^n_h])\commentone{\Delta W_{n}}, \bpsi_h\big)\label{20220608_1}\\
&+k\big(\bF[\bu^n]-\bF[\bu^n_h], \bpsi_h\big).\notag
\end{align}
	
Choosing $\bphi_h=-\opLh\Rh\bfErr{u}{n+\half}$ in \eqref{eqn:prf-fem-form-order-h1-1} yields
\begin{equation}\label{eqn:prf-fem-form-order-h1-2}
\big(\bfErr{u}{n+1}-\bfErr{u}{n}, \opLh\Rh\bfErr{u}{n+\half}\big) = k\big(\bfErr{v}{n+1}, \opLh\Rh\bfErr{u}{n+\half}\big).
\end{equation}
	
Choosing $\bpsi_h=\Ph\bfErr{v}{n+1}$ in \eqref{20220608_1}, we have
\begin{align}\label{eqn:prf-fem-form-order-h1-3}
&\big(\bfErr{v}{n+1}-\bfErr{v}{n}, \Ph\bfErr{v}{n+1}\big) = k\big(\opL\bfErr{u}{n+\half}, \Ph\bfErr{v}{n+1}\big)\\
&\qquad\quad+\big((\bG[\bu^n]-\bG[\bu^n_h])\commentone{\Delta W_{n}}, \Ph\bfErr{v}{n+1}\big)+k\big(\bF[\bu^n]-\bF[\bu^n_h], \Ph\bfErr{v}{n+1}\big).\notag
\end{align}

Based on the definitions of $\Ph$, $\opL$, and $\opLh$, we have
\begin{align}\label{20220613_1}
\big(\opL\bfErr{u}{n+\half}, \Ph\bfErr{v}{n+1}\big)&=\big(\opL\Rh\bfErr{u}{n+\half}, \Ph\bfErr{v}{n+1}\big)\\
&=\big(\opLh\Rh\bfErr{u}{n+\half}, \Ph\bfErr{v}{n+1}\big)\notag\\
&=\big(\opLh\Rh\bfErr{u}{n+\half}, \bfErr{v}{n+1}\big)\notag.
\end{align}

Besides, the left-hand side of \eqref{eqn:prf-fem-form-order-h1-3} can be written as
\begin{align}\label{20220613_2}
&\big(\bfErr{v}{n+1}-\bfErr{v}{n}, \Ph\bfErr{v}{n+1}\big)\\
&\qquad =\big(\Ph\bfErr{v}{n+1}-\Ph\bfErr{v}{n}, \Ph\bfErr{v}{n+1}\big)\notag\\
&\qquad = \frac12\|\Ph\bfErr{v}{n+1}\|_{{\bf L}^2}^2-\frac12\|\Ph\bfErr{v}{n}\|_{{\bf L}^2}^2+\frac12\|\Ph\bfErr{v}{n+1}-\Ph\bfErr{v}{n}\|_{{\bf L}^2}^2\notag.
\end{align}

Combining \eqref{eqn:prf-fem-form-order-h1-2}-\eqref{20220613_2}, we have
\begin{align}\label{eqn:prf-fem-form-order-h1-4}
&J(\Rh\bfErr{u}{n+1}, \Ph\bfErr{v}{n+1})-J(\Rh\bfErr{u}{n}, \Ph\bfErr{v}{n})+\half\exc{\normbf{\Ph\bfErr{v}{n+1}-\Ph\bfErr{v}{n}}}\\
&\qquad=k\big(\bF[\bu^n]-\bF[\bu^n_h], \Ph\bfErr{v}{n+1}\big)+\big((\bG[\bu^n]-\bG[\bu^n_h])\commentone{\Delta W_{n}}, \Ph\bfErr{v}{n+1}\big)\notag\\
&\qquad:=I_1 + I_2.\notag
\end{align}
	
By equation \eqref{20220608_2}, the first term $I_1$ could be bounded by
\begin{align}\label{eqn:prf-fem-form-order-h1-6}
&\exc{k\big(\bF[\bu^n]-\bF[\bu^n_h], \Ph\bfErr{v}{n+1}\big)}\\
&\quad \leq k\exc{\normbf{\bF[\bu^n]-\bF[\bu^n_h]}}+\frac{k}{4}\exc{\normbf{\Ph\bfErr{v}{n+1}}}\notag\\
&\quad \leq kC_B^2\exc{\normdel{\bfErr{u}{n}}+\normbf{\bfErr{u}n}}+\frac{k}{4}\exc{\normbf{\Ph\bfErr{v}{n+1}}}\notag\\
&\quad \leq kC_B^2\exc{\normdel{\Rh\bfErr{u}{n}}+\normbf{\Rh\bfErr{u}n}}\notag\\
&\qquad +kC_B^2\mathbb{E}\Bigl[\normdel{\bu^n-\Rh\bu^n}\notag\\
&\qquad +\normbf{\bu^n-\Rh\bu^n}\Bigr]+\frac{k}{4}\exc{\normbf{\Ph\bfErr{v}{n+1}}}\notag\\
&\quad \leq kCC_B^2\exc{\normdel{\Rh\bfErr{u}{n}}}\notag\\
&\qquad + kC_B^2C_{\Rh}^2h^2\exc{\left\lVert\bu^n\right\rVert^2_{\bH^2}}+\frac{k}{4}\exc{\normbf{\Ph\bfErr{v}{n+1}}}.\notag
\end{align}	
 
Let $M_t=\big((\bG[\bu^n]-\bG[\bu^n_h])\commentone{\Delta W_{n}}, \Rh\bfErr{v}{n}\big)$, $M_t$ is a martingale and $\exc{M_t}=0$. Similar to the estimation of \eqref{eqn:prf-fem-form-order-h1-6}, the second term $I_2$ goes to
\begin{align}\label{eqn:prf-fem-form-order-h1-7}
&\exc{\big((\bG[\bu^n]-\bG[\bu^n_h])\commentone{\Delta W_{n}}, \Ph\bfErr{v}{n+1}\big)}\\
&\qquad\leq k\exc{\normbf{\bG[\bu^n]-\bG[\bu^n_h]}}+\frac{1}{4}\exc{\normbf{\Ph\bfErr{v}{n+1}-\Ph\bfErr{v}{n}}}\notag\\
&\qquad \leq kCC_B^2\exc{\normdel{\Rh\bfErr{u}{n}}}\notag\\
&\qquad\quad+ kC_B^2C_{\Ph}^2h^4\exc{\left\lVert\bu^n\right\rVert^2_{\bH^2}}+\frac14\exc{\normbf{\Ph\bfErr{v}{n+1}-\Ph\bfErr{v}{n}}}.\notag
\end{align}

Combining \eqref{eqn:prf-fem-form-order-h1-4}--\eqref{eqn:prf-fem-form-order-h1-7}, we have
\begin{align}\label{eqn:prf-fem-form-order-h1-8}
&\exc{J(\Rh\bfErr{u}{n+1}, \Ph\bfErr{v}{n+1})}-\exc{J(\Rh\bfErr{u}{n}, \Ph\bfErr{v}{n})}\\
&\quad\leq kCC_B^2\exc{J(\Rh\bfErr{u}{n}, \Ph\bfErr{v}{n})}+ kC_B^2C_{\Ph}^2h^4\exc{\left\lVert\bu^n\right\rVert^2_{\bH^2}}\notag.
\end{align}
		
Then the discrete Gronwall's inequality yields
\begin{align}\label{eqn:prf-fem-form-order-h1-9}
\exc{J(\Rh\bfErr{u}{n+1}, \Ph\bfErr{v}{n+1})}\le C_B^2C_{\Ph}^2h^4\exc{\left\lVert\bu^n\right\rVert^2_{\bH^2}}e^{CC_B^2}.
\end{align}

Combining equations \eqref{eqn:prf-fem-form-order-h1-9}, \eqref{eqn:prj-error-estimate-ph}, and \eqref{20220608_3}, we get the conclusion.
\end{proof}

Theorem \ref{theo:fem-form-order-h1} states that the linear finite element discrete convergence order in $\bH^1$-norm is 1. The next theorem states that the linear finite element discrete convergence order in $\bL^2$-norm is 2.

\begin{theorem}\label{theo:fem-form-order-l2}
Let $\{\bu^n, \bv^n\}_n$ solves \eqref{eqn:time-semiform}--\eqref{20220529_1},  $(\bu^0_h, \bv^0_h)=(\Rh\bu^0, \Ph\bv^0)$ and $(\bu^0, \bv^0)\in \bH^2\times\bH^1$ be given. Assume that $\bF[\bu]$ and $\bG[\bu]$ are linear in $\bu$ and conditions \eqref{20220528_1}--\eqref{20220528_4} hold, then we have 
\begin{equation}\label{eqn:fem-form-order-h2}
\max\limits_{1\leq n \leq N}\exc{\normbf{\bu^n-\bu^n_h}}\leq e^{CC_B^2}\max\limits_{1\le l\le n}\exc{\left\lVert{\bu^{n}}\right\rVert^2_{\bH^2}}h^4.
\end{equation}
\end{theorem}
	
\begin{proof}\label{prf:fem-form-order-l2}
Plugging $\bv^l_h=d_t\bu^l_h:=k^{-1}(\bu^l_h-\bu^{l-1}_h)$ into \eqref{20220608_4}, we have
		
\begin{align}\label{eqn:prf-fem-form-order-l2-1}
\big(d_t\bu^{l+1}_h-d_t\bu^l_h,& \bphi_h\big)+k\lambda\big(\Div(\bu^{l+\half}_h), \Div(\bphi_h)\big) + k\mu\big(\varepsilon(\bu^{l+\half}_h), \varepsilon(\bphi_h)\big)\\
&=\big(\bG[\bu^l_h]\Delta W_{l}, \bphi_h\big) + k\big(\bF[\bu^l_h], \bphi_h\big).\notag
\end{align}
		
Setting $\bar{\bu}^{n}_h=\sum\limits_{l = 1}^{n}\bu^l_h$ and taking the summation give
\begin{align}\label{eqn:prf-fem-form-order-l2-2}
\big(\bu^{n+1}_h-\bu^n_h, \bphi_h\big)=&k^2\big(\opL\bar{\bu}^{n+\half}_h, \bphi_h\big)+\big(\bu^1_h-\bu^0_h, \bphi_h\big)\\
&+k\left(\sum\limits_{l = 1}^{n}\bG[\bu^l_h]\Delta W_{l}, \bphi_h\right) + k^2\left(\sum\limits_{l = 1}^{n}\bF[\bu^l_h], \bphi_h\right).\notag
\end{align}
	
Subtracting \eqref{eqn:prf-fem-form-order-l2-2} from \eqref{eqn:prf-time-semiform-order-l2-2} leads to
\begin{align}\label{eqn:prf-fem-form-order-l2-3}
&\big(\bfErr{u}{n+1}-\bfErr{u}{n}, \bphi_h\big)+k^2\innerprd{\bar{\bf E}_{\bf u}^{n+\half}}{\bphi_h}\\
&\quad =k\big(\sum\limits_{l=1}^{n}(\bG[\bu^l]-\bG[\bu^l_h])\Delta W_{l}, \bphi_h\big) + k^2\left(\sum\limits_{l = 1}^{n}(\bF[\bu^l]-\bF[\bu^l_h]), \bphi_h\right).\notag
\end{align}
	
Choosing $\bphi_h=\Ph\bfErr{u}{n+1}$ in \eqref{eqn:prf-fem-form-order-l2-3}, the first term on the left-hand side of \eqref{eqn:prf-fem-form-order-l2-3} can be written as
\begin{align}\label{20220613_4}
&\big(\bfErr{u}{n+1}-\bfErr{u}{n}, \Ph\bfErr{u}{n+1}\big)\\
&\qquad=\big(\Ph\bfErr{u}{n+1}-\Ph\bfErr{u}{n}, \Ph\bfErr{u}{n+1}\big)\notag\\
&\qquad=\frac12\|\Ph\bfErr{u}{n+1}\|_{{\bf L}^2}^2-\frac12\|\Ph\bfErr{u}{n}\|_{{\bf L}^2}^2+\frac12\|\Ph\bfErr{u}{n+1}-\Ph\bfErr{u}{n}\|_{{\bf L}^2}^2.\notag
\end{align}

Besides, note that
\begin{align}\label{20220615_1}
&\innerprd{\bar{\bf E}_{\bf u}^{n+\half}}{\Ph\bfErr{u}{n+1}}\\
&\qquad=\innerprd{\Rh\bar{\bf E}_{\bf u}^{n+\half}}{\Ph\bfErr{u}{n+1}}\notag\\
&\qquad=-\big(\opLh\Rh\bar{\bf E}_{\bu}^{n+\half}, \Ph\bfErr{u}{n+1}\big)\notag\\
&\qquad=-\big(\opLh\Rh\bar{\bf E}_{\bu}^{n+\half}, \bfErr{u}{n+1}\big)\notag\\
&\qquad=-\big(\opLh\Rh\bar{\bf E}_{\bu}^{n+\half}, \Rh(\bar{\bf E}_{\bu}^{n+1}-\bar{\bf E}_{\bu}^{n})\big)\notag.
\end{align}

Define another energy functional $\cal{E}(\bfErr{u}{n})$ by
\begin{equation}
\cal{E}(\bfErr{u}{n})=\half\normbf{\Ph\bfErr{u}{n}}+\frac{k^2}{2}\left(\normdel{\Rh\bar{\bf E}_{u}^{n}}\right).
\end{equation}
		
Choosing $\bphi_h=\Ph\bfErr{u}{n+1}$ in \eqref{eqn:prf-fem-form-order-l2-3} and using \eqref{20220613_4}--\eqref{20220615_1} yield
\begin{align}\label{eqn:prf-fem-form-order-l2-4}
&\cal{E}(\bfErr{u}{n+1})-\cal{E}({\bfErr{u}{n}})+\half\normbf{\Ph\bfErr{u}{n+1}-\Ph\bfErr{u}{n}}\\
&\qquad=k^2\left(\sum_{l = 1}^{n}(\bF[\bu^l]-\bF[\bu^l_h]), \Ph\bfErr{u}{n+1}\right)\notag\\
&\qquad\quad + k\left(\sum_{l = 1}^{n}(\bG[\bu^l]-\bG[\bu^l_h])\commentone{\Delta W_{n}}, \Ph\bfErr{u}{n+1}\right)\notag\\
&\qquad:=I_1 + I_2.\notag
\end{align}
	
Using the linearity assumption of $F$ and $G$ in $u$,  the $I_1$ term can be bounded by
\begin{align}\label{eqn:prf-fem-form-order-l2-6}
\exc{I_1} &=\exc{k^2\big(\sum_{l = 1}^{n}(\bF[\bu^l]-\bF[\bu^l_h]), \Ph\bfErr{u}{n+1}\big)}\\
&=\exc{k^2\big(\bF[\bar{\bu}^l]-\bF[\bar{\bu}^l_h], \Ph\bfErr{u}{n+1}\big)}\notag\\
&\leq CC_B^2k^3\exc{\normbf{\bar{\bf E}_{u}^{n}}}+k\exc{\normbf{\Ph\bfErr{u}{n+1}}}\notag\\
&\leq CC_B^2k^3\exc{\normdel{\Rh\bar{\bf E}_{\bu}^{n}}}\notag\\
&\quad+CC_B^2kh^4\max\limits_{1\le l\le n}\exc{\left\lVert{\bu^{n}}\right\rVert^2_{\bH^2}}+k\exc{\normbf{\Ph\bfErr{u}{n+1}}}\notag.
\end{align}
	
Let $M_t=k\left(\sum\limits_{l = 0}^{n-1}(\bG[\bu^l]-\bG[\bu^l_h])\commentone{\Delta W_{n}}, \Ph\bfErr{u}{n}\right)$. Then $M_t$ is a martingale and $\exc{M_t}=0$. 
Similarly, using the linearity assumption of $F$ and $G$ in $u$, the $I_2$ term can be bounded by 
\begin{align}\label{eqn:prf-fem-form-order-l2-7}
\exc{I_2}=&\exc{k\left(\sum_{l = 0}^{n-1}\big(\bG[\bu^l]-\bG[\bu^l_h]\big)\commentone{\Delta W_{n}}, \Ph\bfErr{u}{n+1}\right)}\\
&=\exc{k\big((\bG[\bar{\bu}^l]-\bG[\bar{\bu}^l_h])\commentone{\Delta W_{n}}, \Ph\bfErr{u}{n+1}-\Ph\bfErr{u}{n}\big)}\notag\\
&\leq CC_B^2k^3\exc{\normbf{\bar{\bf E}_{u}^{n}}}+\frac14\exc{\normbf{\Ph\bfErr{u}{n+1}-\Ph\bfErr{u}{n}}}\notag\\
&\leq CC_B^2k^3\exc{\normdel{\Rh\bar{\bf E}_{\bu}^{n}}}\notag\\
&\quad+CC_B^2kh^4\max\limits_{1\le l\le n}\exc{\left\lVert{\bu^{n}}\right\rVert^2_{\bH^2}}+\frac14\exc{\normbf{\Ph\bfErr{u}{n+1}-\Ph\bfErr{u}{n}}}\notag.
\end{align}

Combining \eqref{eqn:prf-fem-form-order-l2-4}-\eqref{eqn:prf-fem-form-order-l2-7}, we obtain
\begin{align}\label{eqn:prf-fem-form-order-l2-8}
&\exc{\cal{E}(\bfErr{u}{n+1})}-\exc{\cal{E}(\bfErr{u}{n})}+\frac14\exc{\normbf{\Ph\bfErr{u}{n+1}-\Ph\bfErr{u}{n}}}\\
&\qquad\leq CC_B^2k\exc{\cal{E}(\bfErr{u}{n})}+CC_B^2kh^4\max\limits_{1\le l\le n}\exc{\left\lVert{\bu^{n}}\right\rVert^2_{\bH^2}}.\notag
\end{align}
		
Then the discrete Gronwall's inequality yields
\begin{align}\label{eqn:prf-fem-form-order-l2-9}
\exc{\cal{E}(\bfErr{u}{n+1})}\leq e^{CC_B^2}\max\limits_{1\le l\le n}\exc{\left\lVert{\bu^{n}}\right\rVert^2_{\bH^2}}h^4.
\end{align}
	
The proof is complete.
\end{proof}

\section{Numerical experiments}\label{sec:numerical-tests}
In this section, we present two two-dimensional numerical experiments using the fully discrete Scheme 2 to test and validate our theoretical results. Our computations are done using the software package FEniCS \cite{LangtangenLogg2017}.

{\bf Test 1.} \label{ssec:numerical-test-1}
We consider the following stochastic elastic wave equation with linear multiplicative noise:
\begin{align*}
\left\{
\begin{aligned}
&\bu_{tt}(t,\bf{x})-\Div(\sigma(\bu(t,\bf{x}))) = \bF[\bu]+\bG[\bu]\frac{dW}{dt}\    &\text{ in }&\ \cal{D}\times[0,T],\\
&\bu(t,\bf{x})=\bf{0} &\text{ on }&\ \partial\cal{D}\times[0,T],\\
&\bu(0,\bf{x})=\bu_0,\ \bu_t(0,\bf{x})=\bu_1\ &\text{ in }&\ [0,T],
\end{aligned}
\right.
\end{align*}
where $\commentone{\cal{D}=[0,1]\times[0,1]}$, $\bF[\bu]=|\bu|^2\bu$, and $\bG[\bu] = \commentone{\delta}\bu$. Notice that $\bG$ is linear in $\bu$. 

The datum functions are chosen as
\begin{align*}
\left\{
\begin{aligned}
&\bu_0=\left(\begin{matrix}
\sin^2\pi x_1\sin 2\pi x_2,\\
\sin 2\pi x_1\sin^2\pi x_2
\end{matrix}\right),\\
&\bu_1=-0.3\bu_0.
\end{aligned}
\right.
\end{align*}

Choose the noise intensity parameter $\commentone{\delta} = 0.1$, final time $T=0.5$, and the numerical solution with $k_{ref} = 1/500$ as the numerical exact solution for verifying convergence orders in time. Table \ref{tbl:ex1-err-k} and Figure \ref{fig:ex1-err} list the temporal error and convergence order results of  Test 1.
	
\begin{table}[H]
\caption{Temporal error of Test 1. Uniform meshes with $h = 1/256$ at $T = 0.5$, and $500$ samples are selected.}
\centering
\begin{tabular}{llllllll}
\toprule
\multicolumn{2}{l}{Scheme 2} & \multicolumn{4}{l}{$\exc{\norm{\bu(T)-\bu^N_h}}$}  & \multicolumn{2}{l}{$\exc{\norm{\bv(T)-\bv^N_h}}$}\\
\cmidrule(r){1-2}\cmidrule(r){3-6}\cmidrule(r){7-8}
$h$ & $k$ & $L^2$ error     & order     & $H^1$ error      & order & $L^2$ error     & order\\ \midrule
\multirow{5}{1.5cm}{$1/256$} 
& 1/50 & $\Exp{9.84}{-3}$ & —— & $\Exp{8.11}{-2}$ & ——  & $\Exp{1.69}{-2}$ & —— \\
& 1/75 & $\Exp{6.88}{-3}$ & 0.88 & $\Exp{6.80}{-2}$ & 0.44 & $\Exp{1.46}{-2}$ & 0.37 \\
& 1/100 & $\Exp{5.17}{-3}$ & 0.99 & $\Exp{5.97}{-2}$ & 0.45 & $\Exp{1.30}{-2}$ & 0.39 \\
& 1/150 & $\Exp{3.43}{-3}$ & 1.01 & $\Exp{4.87}{-2}$ & 0.50 & $\Exp{1.09}{-2}$ & 0.44\\
& 1/200 & $\Exp{2.51}{-3}$ & 1.09 & $\Exp{4.16}{-2}$ & 0.55 & $\Exp{9.57}{-3}$ & 0.45 \\ \bottomrule
\end{tabular}\label{tbl:ex1-err-k}
\end{table}
	
\begin{figure}[H]
\centering
\includegraphics[height=2.5in,width=0.6\linewidth]{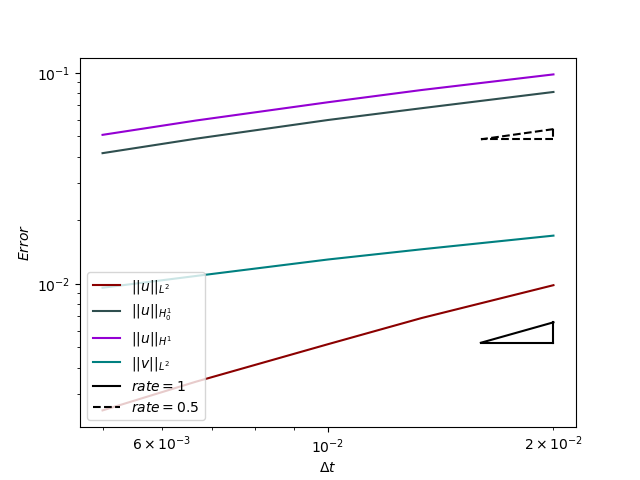}
\caption{\small{Rate of the tempporal error convergence of Test 1 in $\bL^2$- and $\bH^1$-norm. Here $T = 0.5$, $h = 1/256$, and $k\in \{50^{-1}, 75^{-1}, 100^{-1}, 150^{-1}, 200^{-1}\}.$}}\label{fig:ex1-err}
\end{figure}
	
Now we fix $k_{ref} = 1/500$ and choose $h_{ref} = 1/512$. Table \ref{tbl:ex1-err-h} and Figure \ref{fig:ex1-err-h} list the spatial error and convergence order results of Test 1.
\begin{table}[H]
\caption{Spatial error of Test 1. The reference solution is given by $h_{ref} = 1/512$ and $500$ samples are selected.}
\centering
\begin{tabular}{llllll}
\toprule
\multicolumn{2}{l}{Scheme 2} & \multicolumn{4}{l}{$\exc{\norm{\bu(T)-\bu^N_h}}$}\\
$k$ & $h$ & $L^2$ error     & order     & $H^1_0$ error      & order \\
\midrule
\multirow{3}{1.5cm}{$1/500$} 
& 1/64 & $\Exp{9.38}{-5}$ & —— & $\Exp{1.76}{-3}$ & ——\\
& 1/128 & $\Exp{2.37}{-5}$ & 1.98 & $\Exp{9.52}{-4}$ & 0.89\\
& 1/256 & $\Exp{5.71}{-6}$ & 2.06 & $\Exp{4.44}{-4}$ & 1.10\\
\bottomrule
\end{tabular}\label{tbl:ex1-err-h}
\end{table}
	
\begin{figure}[H]
\centering
\includegraphics[height=2.5in,width=0.6\linewidth]{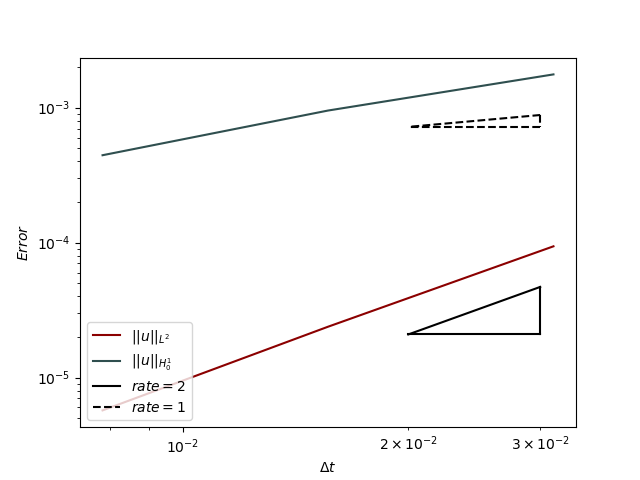}
\caption{\small{The spatial convergence rate of Test 1 in $\bL^2$- and $\bH^1$-norm. Here $T = 0.5$, $k = 1/500$, and $h\in \{2^{(-6+i)}, i = 0, 1, 2\}$.}}
\label{fig:ex1-err-h}
\end{figure}

Finally, to verify the stability of the proposed numerical methods, we compute the energy norm numerically and draw the graph of the energy norm over time, as shown in Figure \ref{fig:ex1-energy}.
		
\begin{figure}
\centering
\subfigure[$k = 1/50$]{\includegraphics[width=0.48\columnwidth]{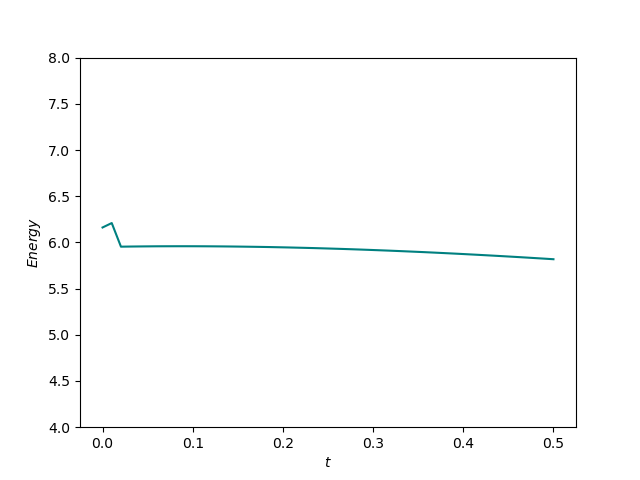}} 
\subfigure[$k = 1/75$]{\includegraphics[width=0.48\columnwidth]{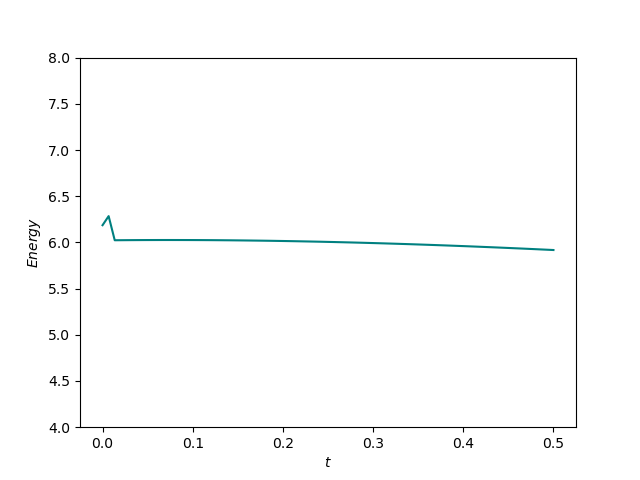}} \\
\subfigure[$k = 1/100$]{\includegraphics[width=0.48\columnwidth]{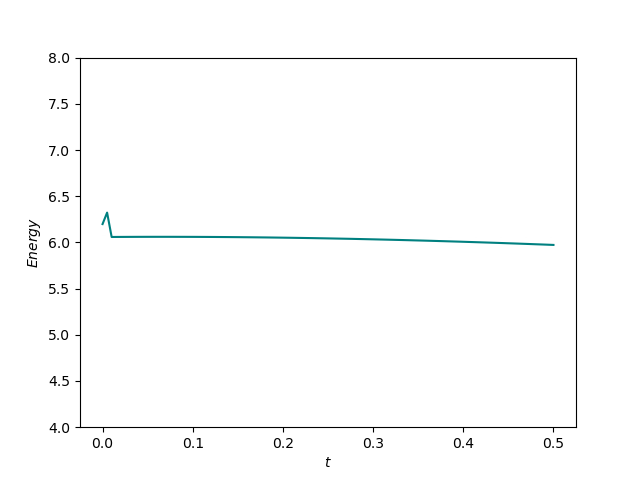}} 
\subfigure[$k = 1/150$]{\includegraphics[width=0.48\columnwidth]{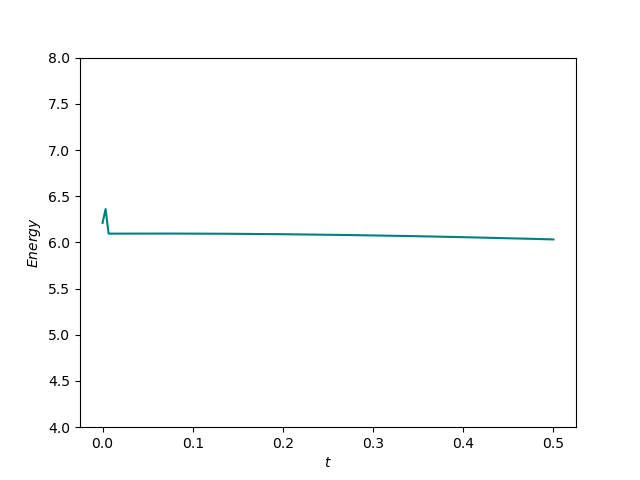}} \\
\subfigure[$k = 1/200$]{\includegraphics[width=0.48\columnwidth]{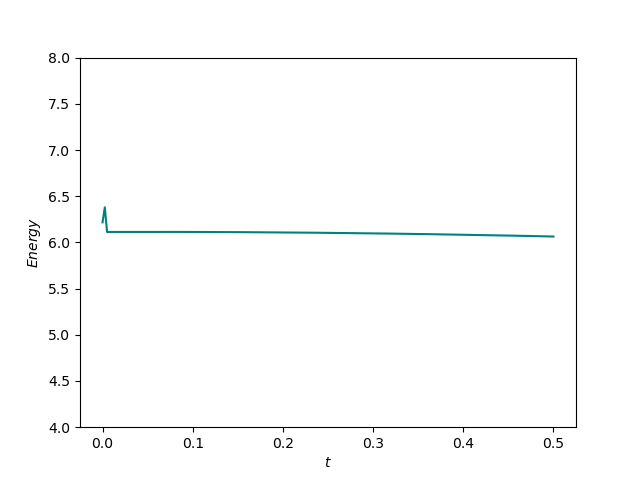}} 
\subfigure[$k = k_{ref}$]{\includegraphics[width=0.48\columnwidth]{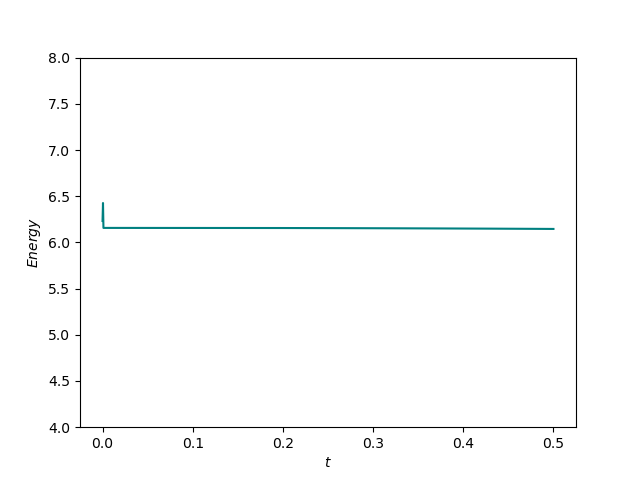}}
\caption{\small\itshape{Energy norm in Test 1 over time based on different time step sizes.}}
\label{fig:ex1-energy}
\end{figure}
	
{\bf Test 2.} \label{ssec:numerical-test-2}
We consider the stochastic elastic wave equations with nonlinear multiplicative noise as given below. 
\begin{align*}
\left\{
\begin{aligned}
&\bu_{tt}(t,\bf{x})-\Div(\sigma(\bu(t,\bf{x}))) = \bF[\bu]+\bG[\bu]\frac{dW}{dt}\    &\text{in}&\ \cal{D}\times[0,T],\\
&\bu(t,\bf{x})=\bf{0} &\text{on}&\ \partial\cal{D}\times[0,T],\\
&\bu(0,\bf{x})=\bu_0,\ \bu_t(0,\bf{x})=\bu_1\ &\text{in}&\ [0,T],\\
\end{aligned}
\right.
\end{align*}
where $\commentone{\cal{D}=[0,1]\times[0,1]}$, $\bF[\bu]=|\bu|^2\bu$ and $\bG[\bu] = \commentone{\delta}(|\bu|^2+1)\bu$. Notice that 
both $\bF$ and $\bG$ are nonlinear in $\bu$. 

The datum functions are chosen as
\begin{align*}
\left\{
\begin{aligned}
&\bu_0=\left(\begin{matrix}
\sin 3\pi x_1\sin 2\pi x_2,\\
\sin 2\pi x_1\sin 3\pi x_2
\end{matrix}\right),\\
&\bu_1=\bf{0}.
\end{aligned}
\right.
\end{align*}

Choose the noise intensity parameter $\commentone{\delta} = 0.1$, final time $T=0.5$, and the numerical solution with $k_{ref} = 1/500$ as the numerical exact solution for verifying the time convergence order. Table \ref{tbl:ex2-err-k} and Figure \ref{fig:ex2-err} list the temporal error and convergence order results of Test 1.
	
\begin{table}[H]
\caption{Temporal error of Test 2. Uniform mesh with $h = 1/256$ at $T = 0.5$, and $500$ samples are selected.}
\centering
\begin{tabular}{llllllll}
\toprule
\multicolumn{2}{l}{Scheme 2} & \multicolumn{4}{l}{$\exc{\norm{\bu(T)-\bu^N_h}}$}  & \multicolumn{2}{l}{$\exc{\norm{\bv(T)-\bv^N_h}}$}\\
\cmidrule(r){1-2}\cmidrule(r){3-6}\cmidrule(r){7-8}
$h$ & $k$ & $L^2$ error     & order     & $H^1_0$ error      & order & $L^2$ error     & order\\ \midrule
\multirow{5}{1.5cm}{$1/256$} 
& 1/50 & $\Exp{1.32}{-1}$ & —— & $\Exp{1.65}{0}$ & ——  & $\Exp{4.42}{-1}$ & —— \\
& 1/75 & $\Exp{9.46}{-2}$ & 0.82 & $\Exp{1.41}{0}$ & 0.39 & $\Exp{3.76}{-1}$ & 0.40 \\
& 1/100 & $\Exp{7.27}{-2}$ & 0.92 & $\Exp{1.24}{0}$ & 0.43 & $\Exp{3.32}{-1}$ & 0.43 \\
& 1/150 & $\Exp{4.88}{-2}$ & 0.98 & $\Exp{1.03}{0}$ & 0.47 & $\Exp{2.79}{-1}$ & 0.43 \\
& 1/200 & $\Exp{3.65}{-2}$ & 1.01 & $\Exp{8.96}{-1}$ & 0.50 & $\Exp{2.41}{-1}$ & 0.51 \\ \bottomrule
\end{tabular}\label{tbl:ex2-err-k}
\end{table}
	
\begin{figure}[H]
\centering
\includegraphics[height=2.5in,width=0.6\linewidth]{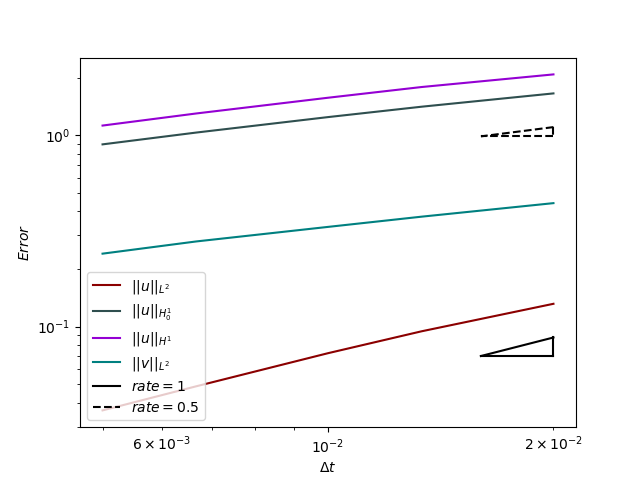}
\caption{\small{Rate of the temporal error convergence of Test 2 in $\bL^2$- and $\bH^1$-norm. Here $T = 0.5$, $h = 1/256$, and $k\in \{50^{-1}, 75^{-1}, 100^{-1}, 150^{-1}, 200^{-1}\}$.}}\label{fig:ex2-err}
\end{figure}
	
Now we fix $k_{ref} = 1/500$ and choose $h_{ref} = 1/512$. Table \ref{tbl:ex2-err-h} and Figure \ref{fig:ex2-err-h} list the spatial error and convergence order results of Test 2.
	
\begin{table}[H]
\caption{Spatial error of Test 2. The reference solution is computed by using  $h_{ref} = 1/512$  and $500$ samples are selected.}
		
\centering
\begin{tabular}{llllll}
\toprule
\multicolumn{2}{l}{Scheme 2} & \multicolumn{4}{l}{$\exc{\norm{\bu(T)-\bu^N_h}}$}\\
\cmidrule(r){1-2}\cmidrule(r){3-6}
$k$ & $h$ & $L^2$ error     & order     & $H^1_0$ error      & order \\
\midrule
\multirow{3}{1.5cm}{$1/500$} 
& 1/64 & $\Exp{7.56}{-4}$ & —— & $\Exp{3.10}{-2}$ & ——\\
& 1/128 & $\Exp{2.30}{-4}$ & 1.72 & $\Exp{1.77}{-2}$ & 0.81\\
& 1/256 & $\Exp{5.94}{-5}$ & 1.95 & $\Exp{9.10}{-3}$ & 0.96\\
\bottomrule
\end{tabular}\label{tbl:ex2-err-h}
\end{table}
	
\begin{figure}[H]
\centering
\includegraphics[width=0.6\linewidth]{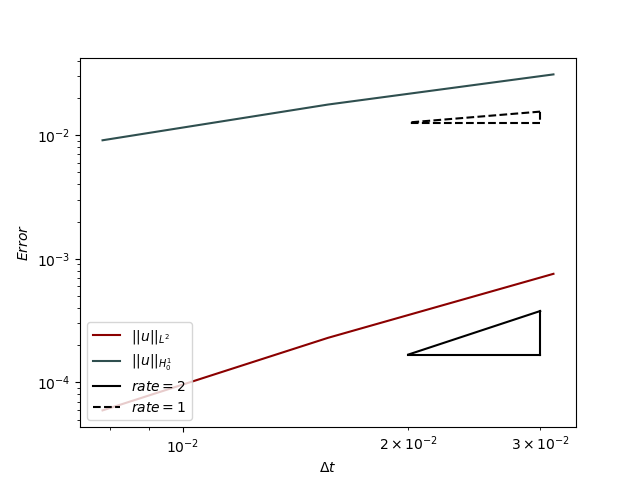}
\caption{\small{Rate of the spatial error convergence of Test 2 in the $\bL^2$- and  $\bH^1$-norm. $T = 0.5$, $k = 1/500$ and $h\in \{2^{(-6+i)}, i = 0, 1, 2\}.$}}
\label{fig:ex2-err-h}
\end{figure}

Finally, to verify the stability of the proposed numerical method, we compute the energy norm numerically and draw the graph of the energy norm over time, as shown in Figure \ref{fig:ex2-energy}.
	
\begin{figure}
\centering
\subfigure[$k = 1/50$]{\includegraphics[width=0.48\columnwidth]{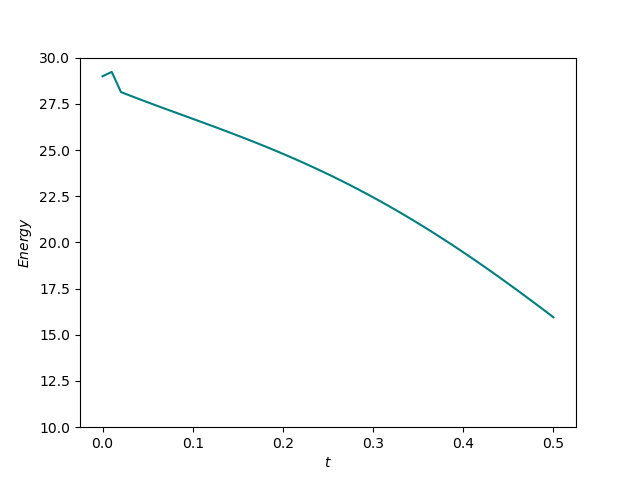}} 
\subfigure[$k = 1/75$]{\includegraphics[width=0.48\columnwidth]{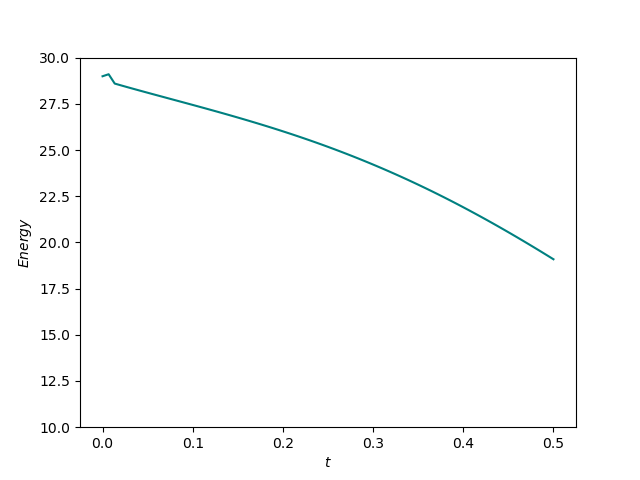}} \\
\subfigure[$k = 1/100$]{\includegraphics[width=0.48\columnwidth]{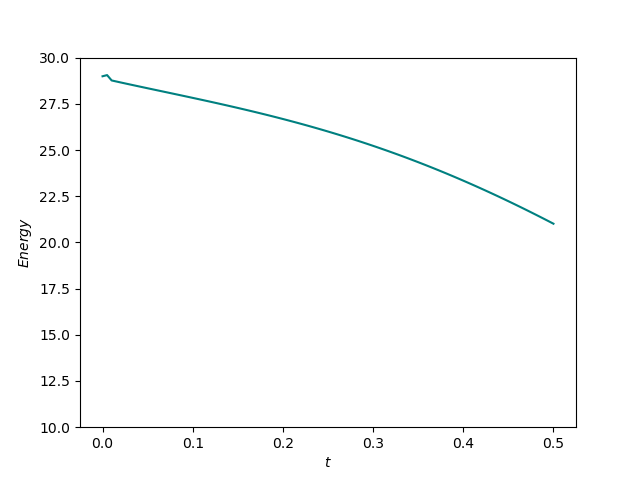}} 
\subfigure[$k = 1/150$]{\includegraphics[width=0.48\columnwidth]{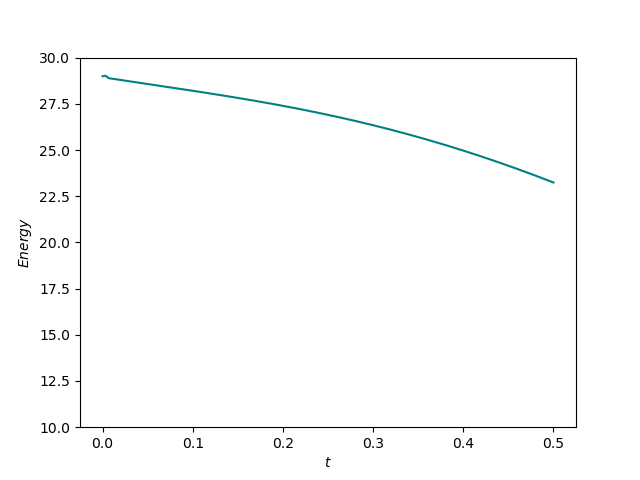}} \\
\subfigure[$k = 1/200$]{\includegraphics[width=0.48\columnwidth]{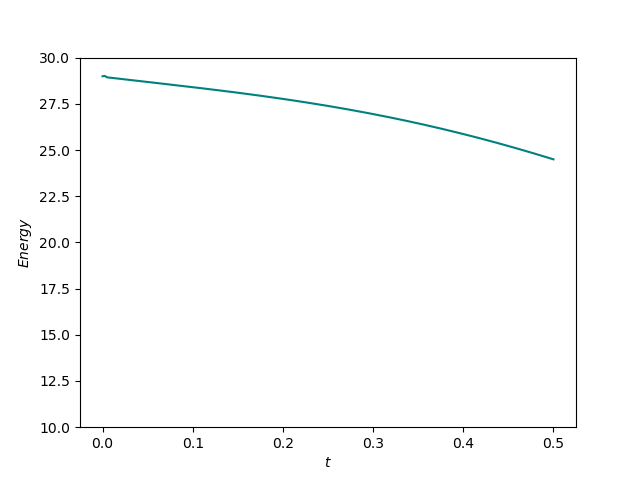}} 
\subfigure[$k = \Delta t_{ref}$]{\includegraphics[width=0.48\columnwidth]{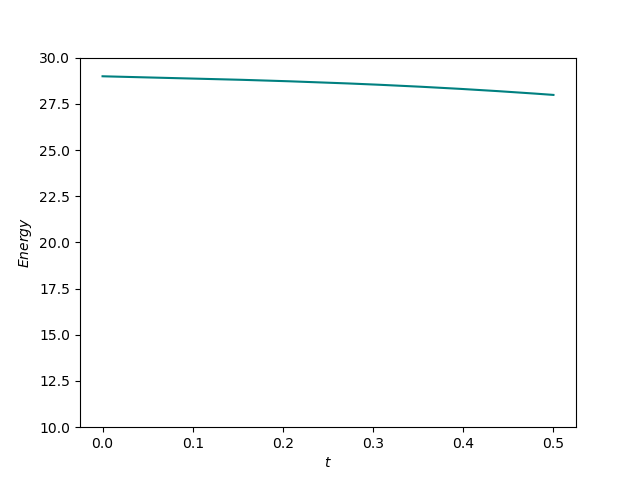}}
\caption{\small\itshape{Energy norm of Test 2 over time, using different time step sizes.}}
\label{fig:ex2-energy}
\end{figure}

\section{Conclusion}\label{sec:conclustion}
In this paper we proposed a semi-discrete in time scheme and a fully discrete finite element method for nonlinear stochastic elastic wave equations with multiplicative noise. We proved various properties for the strong solution such as stability and H\"older continuity estimates. We
also established the stability and error estimates for the semi-discrete numerical solution, which show that the $\bL^2$-norm of the temporal error has first order convergence, while the $\bH^1$-norm of the error has one half order convergence. Moreover, we proved that, for the linear finite element, the $\bL^2$-norm of the spatial error has second order convergence, and the $\bH^1$-norm of the error has first order convergence. Two-dimensional numerical experiments were also presented using the proposed numerical methods to validate the theoretical results proved in the paper.
	
To the best of our knowledge, no numerical analysis result for stochastic elastic wave equations has been reported in the literature. The work of this paper fills the void in this area.  Like in the 
numerical analysis of stochastic parabolic PDEs, a key idea and technique for overcoming the difficulty caused by the noise is to establish and make use of the H\"older continuity in time 
of the weak solution in various spatial norms. 
We also note that the results of this paper could be improved in light of the recent work \cite{feng2022higher}. In particular, it is possible to construct better time-stepping schemes which can achieve the optimal $\frac32$ order of convergence for the 
displacement approximation, which in turn requires higher order approximation for the multiplicative noise term. In addition, the classical Monte Carlo (MC) method is computationally inefficient, more efficient quasi-MC methods could be used to speed up the computation of
the expected values. Moreover, other numerical methods such as the stochastic spectral method \cite{zhang2017numerical} could be used to process and analyze the noise term. Finally, the analysis techniques of this paper may not be able to handle  more general nonlinear functions $\bF[\bu]$ and $\bG[\bu]$,  new analysis techniques are needed for the job. We plan to continue addressing those issues in a future work. 

\bigskip
\noindent
{\bf Acknowledgment.} The first author was partially supported by the National Natural Science Foundation under Grant No. DMS-2012414, and the second author was partially supported by the National Natural Science Foundation under Grant No. DMS-2110728.

\end{document}